\documentclass[twoside,12pt]{amsart}
\usepackage{amsmath,amssymb,amsthm,longtable,graphicx}
\usepackage{txfonts,multirow}
\usepackage[alphabetic,abbrev]{amsrefs}

\newsavebox{\circlebox}
\savebox{\circlebox}{\fontencoding{OMS}\selectfont\Large\char13}
\newlength{\circleboxwdht}
\newcommand{\centercircle}[1]{%
  \setlength{\circleboxwdht}{\wd\circlebox}%
  \addtolength{\circleboxwdht}{\dp\circlebox}%
  \raisebox{0.4\dp\circlebox}{%
    \parbox[][\circleboxwdht][c]{\wd\circlebox}{\centering#1}}%
  \llap{\usebox{\circlebox}}%
}

\address{Graduate School of Mathematics, Nagoya University, Chikusa-ku, 
Nagoya 464-8602, Japan}
\email{kazuto.iijima@math.nagoya-u.ac.jp}
\thanks{}
\dedicatory{}

\newtheorem{definition}{Definition}[section]
\newtheorem{theorem}[definition]{Theorem}
\newtheorem{proposition}[definition]{Proposition}
\newtheorem{example}[definition]{Example}
\newtheorem{corollary}[definition]{Corollary}
\newtheorem{lemma}[definition]{Lemma}

\newtheorem{claim}{Claim}

\begin{document}

\title[]{A comparison of $q$-decomposition numbers in the $q$-deformed Fock spaces of higher levels}
\author[K.~IIJIMA]{Kazuto Iijima}
\date{}
\maketitle

\thispagestyle{empty}

\begin{abstract}
The $q$-deformed Fock spaces of higher levels were introduced by Jimbo-Misra-Miwa-Okado. 
The $q$-decomposition matrix is a transition matrix from the standard basis to the canonical basis 
	defined by Uglov in the $q$-deformed Fock space. 
In this paper, we show that parts of $q$-decomposition matrices of level $\ell$ coincides with that of level $\ell - 1$ 
	under certain conditions of multi charge. 

\end{abstract}

%
%

\section{Introduction}

The $q$-deformed Fock spaces of higher levels were introduced by Jimbo-Misra-Miwa-Okado \cite{JMMO}. 
For a multi charge $\boldsymbol{s} = (s_{1} , \ldots , s_{\ell}) \in \mathbb{Z}^{\ell}$, 
the $q$-deformed Fock space $\boldsymbol{F}_{q}[\boldsymbol{s}]$ of level $\ell$ is 
the $\mathbb{Q}(q)$-vector space whose basis are indexed by $\ell$-tuples of Young diagrams. i.e. 
$\{ | \boldsymbol{\lambda} ; \boldsymbol{s} \rangle \, | \, \boldsymbol{\lambda} \in \Pi^{\ell} \}$, 
where $\Pi$ is the set of Young diagrams.
Heisenberg group (resp. quantum group $U_{q}(\hat{\mathfrak{sl_{n}}})$) acts on $\boldsymbol{F}_{q}[\boldsymbol{s}]$ 
	as level $q^{n \ell}$ (resp. level $q^{\ell}$). 
Both actions commute on $\boldsymbol{F}_{q}[\boldsymbol{s}]$. 

The canonical bases 
$\{ G^{+}(\boldsymbol{\lambda} ; \boldsymbol{s}) \, | \, \boldsymbol{\lambda} \in \Pi^{\ell} \}$ and 
$\{ G^{-}(\boldsymbol{\lambda} ; \boldsymbol{s}) \, | \, \boldsymbol{\lambda} \in \Pi^{\ell} \}$ are 
bases of the Fock space $\boldsymbol{F}_{q}[\boldsymbol{s}]$ 
that are invariant under a certain involution $\overline{\phantom{a}}$ \cite{U}.
Define matrices  
$\Delta^{+}(q) = (\Delta^{+}_{\boldsymbol{\lambda},\boldsymbol{\mu}}(q))
_{\boldsymbol{\lambda},\boldsymbol{\mu}}$ and 
$\Delta^{-}(q) = (\Delta^{-}_{\boldsymbol{\lambda},\boldsymbol{\mu}}(q))
_{\boldsymbol{\lambda},\boldsymbol{\mu}}$ by 

\begin{align*}
G^{+}(\boldsymbol{\lambda} ; \boldsymbol{s}) = \sum_{\boldsymbol{\mu}} 
\Delta^{+}_{\boldsymbol{\lambda},\boldsymbol{\mu}}(q) \, | \, \boldsymbol{\mu} ; \boldsymbol{s} \rangle 
\hspace{2em} , \hspace{3em} 
G^{-}(\boldsymbol{\lambda} ; \boldsymbol{s}) = \sum_{\boldsymbol{\mu}} 
\Delta^{-}_{\boldsymbol{\lambda},\boldsymbol{\mu}}(q) \, | \, \boldsymbol{\mu} ; \boldsymbol{s} \rangle .
\end{align*}
We call $\Delta^{+}_{\boldsymbol{\lambda},\boldsymbol{\mu}}(q)$ and $\Delta^{-}_{\boldsymbol{\lambda},\boldsymbol{\mu}}(q)$ 
{\it $q$-decomposition numbers}. 
These $q$-decomposition matrices play an important role in representation theory. 
However it is difficult to compute $q$-decomposition matrices. 

In the case of $\ell = 1$, 
Varagnolo-Vasserot \cite{VV1} proved that $\Delta^{+}(1)$ coincides with the decomposition matrix of $v$-Schur algebra. 
Ariki defined a $q$-analogue of decomposition numbers of $v$-Schur algebra by using Khovanov-Lauda's grading, 
	and proved that it coincides with the $q$-decomposition numbers \cite{A}. 
For $\ell \ge 2$, Yvonne \cite{Y} conjectured that the matrix $\Delta^{+}(q)$ coincides with 
the $q$-analogue of the decomposition matrices of cyclotomic Schur algebras at a primitive $n$-th root of unity 
under a suitable condition on multi charge. 

Let $\mathcal{O}_{\boldsymbol{s}}( \ell , 1, m)$ be the category $\mathcal{O}$ of rational Cherednik algebra of 
$( \mathbb{Z} / \ell \mathbb{Z} ) \wr \mathfrak{S}_{m}$ associated with multicharge $\boldsymbol{s}$. 
Rouquier \cite[Theorem 6.8, \S 6.5]{R2} conjectured that, for arbitrary multi charge, 
the multiplicities of simple modules in standard modules in $\mathcal{O}_{\boldsymbol{s}}( \ell , 1, m)$  
are equal to the corresponding coefficients $\Delta^{+}_{\boldsymbol{\lambda},\boldsymbol{\mu}}(q)$, 
where $m = | \boldsymbol{\lambda} | = | \boldsymbol{\mu} |$. 
It is expected that $\oplus_{m \ge 0} \mathcal{O}_{\boldsymbol{s}}( \ell , 1, m)$ 
	should categorify $\boldsymbol{F}_{1}[\boldsymbol{s}]$. 
(see \cite{S} for the details.) 
More generally, it is expected that, together with a suitable grading, $\oplus_{m \ge 0} \mathcal{O}_{\boldsymbol{s}}( \ell , 1, m)$ 
	should categorify $\boldsymbol{F}_{q}[\boldsymbol{s}]$. 
For the detail of correspondence between the charges of $\mathcal{O}_{\boldsymbol{s}}( \ell , 1, m)$ and 
the charges of Fock spaces, see \cite{R2}. 

Now, we state our main theorems. 
We say that 
the $j$-th component $s_{j}$ of the multi charge is {\it sufficiently large} for $| \boldsymbol{\lambda} ; \boldsymbol{s} \rangle$ 
if $s_{j} - s_{i} \ge \lambda_{1}^{(i)}$ for any $i=1,2,\cdots, \ell$, 
and that $s_{j}$ is {\it sufficiently small} for $| \boldsymbol{\lambda} ; \boldsymbol{s} \rangle$ if 
$s_{i} - s_{j} \geq | \boldsymbol{\lambda} | = |\lambda^{(1)}| + \cdots + |\lambda^{(\ell)}| $ for any $i=1,2,\cdots, \ell$
(see Definition \ref{largesmall}). 
More generally, for a positive integer $N$ we say that $s_{j}$ is sufficiently small for $N$ if 
$s_{i} - s_{j} \geq N$ for all $i \not= j $.
If $s_{j}$ is sufficiently large for $ | \boldsymbol{\lambda} ; \boldsymbol{s} \rangle $ and 
$ | \boldsymbol{\lambda} ; \boldsymbol{s} \rangle > | \boldsymbol{\mu} ; \boldsymbol{s} \rangle$, 
then the $j$-th components of $\boldsymbol{\lambda}$ and $\boldsymbol{\mu}$ are 
both the empty Young diagram $\emptyset$ (Lemma \ref{lem.3.2}). 
On the other hand, if $s_{j}$ is sufficiently small for $|\boldsymbol{\lambda} ; \boldsymbol{s} \rangle$ and 
$|\boldsymbol{\lambda} ; \boldsymbol{s} \rangle \ge |\boldsymbol{\mu} ; \boldsymbol{s} \rangle$, then  
$\mu^{(j)} = \emptyset$ implies $\lambda^{(j)} = \emptyset$.
(Lemma \ref{lem9}). 

Our main results are as follows: 

\smallskip

\underline{ \bf{Theorem A.} \rm{} (Theorem \ref{thmA}) }

Let $\varepsilon \in \{ +,- \}$. 
If $s_{j}$ is sufficiently large for $|\boldsymbol{\lambda} ; \boldsymbol{s} \rangle$, then 

\begin{equation*}
\Delta^{\varepsilon}_{ \boldsymbol{\lambda} , \boldsymbol{\mu} ; \boldsymbol{s} }(q)
=\Delta^{\varepsilon}_{ \check{\boldsymbol{\lambda}} , \check{\boldsymbol{\mu}} ; \check{\boldsymbol{s}} }(q),
\end{equation*}
where $\check{\boldsymbol{\lambda}}$ (resp. $\check{\boldsymbol{\mu}} , \check{\boldsymbol{s}})$ 
is obtained by omitting the $j$-th component of $\boldsymbol{\lambda}$ 
(resp. $\boldsymbol{\mu} , \boldsymbol{s})$, 
$\Delta^{\varepsilon}_{ \boldsymbol{\lambda} , \boldsymbol{\mu} ; \boldsymbol{s} }(q)$ 
is the $q$-decomposition number of level $\ell$ and 
$\Delta^{\varepsilon}_{ \check{\boldsymbol{\lambda}} , \check{\boldsymbol{\mu}} ; \check{\boldsymbol{s}} }(q)$ 
is the $q$-decomposition number of level $\ell -1$. 

\smallskip

\underline{ \bf{Theorem B.} \rm{} (Theorem \ref{thmB}) }

Let $\varepsilon \in \{ +,- \}$. 
If $s_{j}$ is sufficiently small for $|\boldsymbol{\mu} ; \boldsymbol{s} \rangle$ and $\mu^{(j)} = \emptyset$, then 

\begin{equation*}
\Delta^{\varepsilon}_{ \boldsymbol{\lambda} , \boldsymbol{\mu} ; \boldsymbol{s} }(q)
=\Delta^{\varepsilon}_{ \check{\boldsymbol{\lambda}} , \check{\boldsymbol{\mu}} ; \check{\boldsymbol{s}} }(q),
\end{equation*}
where $\check{\boldsymbol{\lambda}}$ (resp. $\check{\boldsymbol{\mu}} , \check{\boldsymbol{s}})$ 
is obtained by omitting the $j$-th component of $\boldsymbol{\lambda}$ 
(resp. $\boldsymbol{\mu} , \boldsymbol{s})$.

\vspace{1em}

Shoji and Wada proved some product formulae of $q$-decomposition numbers \cite[Theorem 2.9]{SW}. 
There are some overlaps between our results and their product formula.
 \cite{SW} has some assumptions ``dominance'' on the multi charge while our results don't. 
On the concluding facts, \cite{SW} has a flexibility of embedding of $q$-decomposition matrices while our results don't.

Our results are related to category $\mathcal{O}$ in the following sense. 
In the category $\mathcal{O}$, Chuang and Miyachi conjectured the following: 
\smallskip

\underline{ \bf{Conjectures.} \cite[\S 5]{CM}}

\begin{itemize}
\item[(A')] 
Let $\boldsymbol{\lambda}' \in \Pi^{\ell}$. 
If $s_{1}$ is sufficiently large for any $| ( \emptyset , \boldsymbol{\lambda}') ; \boldsymbol{s} \rangle$, 
there exists an embedding 
\begin{equation*}
	\mathcal{O}_{\check{\boldsymbol{s}}}( \ell , 1, m) \hookrightarrow \mathcal{O}_{\boldsymbol{s}}( \ell +1 , 1, m) .
\end{equation*}
\item[(B')] 
If $s_{\ell}$ is sufficiently small for $m$, 
there exists a quotient functor 
\begin{equation*}
	\mathcal{O}_{\boldsymbol{s}}( \ell +1 , 1, m) \twoheadrightarrow \mathcal{O}_{\check{\boldsymbol{s}}}( \ell , 1, m) , 
\end{equation*}
where $\check{\boldsymbol{s}}$ is obtained by omitting the $j$-th component of $\boldsymbol{s}$.  
\end{itemize}

\vspace{1em}

We see that Conjecture (A') (resp. (B')) is consistent with Theorem A (resp. Theorem B) by taking into account 
	the conjecture that $\oplus_{m \ge 0} \mathcal{O}_{\boldsymbol{s}}( \ell , 1, m)$ 
	should categorify $\boldsymbol{F}_{q}[\boldsymbol{s}]$. 
Theorem A (resp. Theorem B) gives a strong support to the conjecture (A') (resp. (B')).

This paper is organized as follows. 
In Section 2, we review the $q$-deformed Fock spaces of higher levels and its canonical bases. 
In Section 3, we state the main results. 
In Section 4, we review the straightening rules in the $q$-deformed Fock spaces.  
Theorem A(Theorem \ref{thmA}) and Theorem B(Theorem \ref{thmB}) are proved in Section 5 and 6 respectively. 

\subsection*{Acknowledgments}

I am deeply grateful to Hyohe Miyachi and Soichi Okada for their advice.

%
%

\subsection*{Notations}

For a positive integer $N$, a \it{partition} \rm{of} $N$ is 
a non-increasing sequence of non-negative integers summing to $N$. 
We write $|\lambda| = N$ if $\lambda$ is a partition of $N$.
The \it{length} $l(\lambda)$ \rm{of} $\lambda$ is the number of non-zero components of $\lambda$.
And we use the same notation $\lambda$ to represent the Young diagram corresponding to $\lambda$.
For an $\ell$-tuple $\boldsymbol{\lambda} = (\lambda^{(1)} , \lambda^{(2)} , \cdots , \lambda^{(\ell)})$ of Young diagrams, 
	we put $|\boldsymbol{\lambda}| = |\lambda^{(1)}| + |\lambda^{(2)}| + \cdots + |\lambda^{(\ell)}|$.

%
%

\section{The $q$-deformed Fock spaces of higher levels}

\subsection{$q$-wedge products and straightening rules}

Let $n$, $\ell$, $s$ be integers such that $n \ge 2$ and $\ell \ge 1$. 
We define $P(s)$ and $P^{++}(s)$ as follows;

\begin{align}
P(s) &= \{ \boldsymbol{k} = ( k_{1} , k_{2} , \cdots ) \in \mathbb{Z}^{\infty}  \,\, | \,\, 
	k_{r} = s-r+1 \,\, \text{ for any sufficiently large } r \,\, \}  \\
P^{++}(s) &= \{ \boldsymbol{k} = ( k_{1} , k_{2} , \cdots ) \in P(s)  \,\, | \,\, k_{1} > k_{2} > \cdots \,\, \}  .
\end{align}

Let $\Lambda^{s}$ be the $\mathbb{Q}(q)$ vector space spanned by the $q$-wedge products 
\begin{equation}
u_{\boldsymbol{k}} = u_{k_{1}} \wedge u_{k_{2}} \wedge \cdots \,\,\,\, , \,\,\,\, (\boldsymbol{k} \in P(s))
\end{equation}
subject to certain commutation relations, so-called straightening rules. 
Note that the straightening rules depend on $n$ and $\ell$. 
\cite[Proposition 3.16]{U} (The precise description will be given in \S 4.)

\begin{example}
$(i)$ For every $k_{1} \in \mathbb{Z}$, $u_{k_{1}} \wedge u_{k_{1}} = - u_{k_{1}} \wedge u_{k_{1}}$. 
Therefore $u_{k_{1}} \wedge u_{k_{1}} = 0$.

$(ii)$ Let $n=2$, $\ell =2$, $k_{1}=-2$, and $k_{2}=4$. 
Then 
$$ u_{-2} \wedge u_{4} = q \, u_{4} \wedge u_{-2} + (q^{2} - 1) \, u_{2} \wedge u_{0}. $$

$(iii)$ Let $n=2$, $\ell =2$, $k_{1}=-1$, $k_{2}=-2$ and $k_{3}=4$. 
Then 
\begin{align*}
u_{-1} \wedge u_{-2} \wedge u_{4} = u_{-1} \wedge ( u_{-2} \wedge u_{4} ) 
&= u_{-1} \wedge \Big( q \, u_{4} \wedge u_{-2} + (q^{2} - 1) \, u_{2} \wedge u_{0} \Big) \\
&= q \, u_{-1} \wedge u_{4} \wedge u_{-2} + (q^{2} - 1) \, u_{-1} \wedge u_{2} \wedge u_{0} 
\end{align*}
\label{ex3}
\end{example}

By applying the straightening rules, every $q$-wedge product $u_{\boldsymbol{k}}$ is 
expressed as a linear combination of so-called 
\it{ordered $q$-wedge products}, \rm{namely} $q$-wedge products $u_{\boldsymbol{k}}$ with $\boldsymbol{k} \in P^{++}(s)$.
The ordered $q$-wedge products $\{ u_{\boldsymbol{k}}  \,\, | \,\, \boldsymbol{k} \in P^{++}(s) \}$ 
form a basis of $\Lambda^{s}$ called \it{the standard basis.} \rm{}

\subsection{Abacus}

It is convenient to use the abacus notation for studying various properties in straightening rules. 

Fix an integer $N \ge 2$, and form an infinite abacus with $N$ runners labeled $1,2,\cdots N$ from left to right. 
The positions on the $i$-th runner are labeled by the integers having residue $i$ modulo $N$. 

\begin{equation*}
\begin{array}{ccccc}
\vdots & \vdots & \vdots & \vdots & \vdots \\
-N+1 & -N+2 & \cdots & -1 & 0 \\
1 & 2 & \cdots & N-1 & N \\
N+1 & N+2 & \cdots & 2N-1 & 2N \\
\vdots & \vdots & \vdots & \vdots & \vdots
\end{array}
\end{equation*}

Each $\boldsymbol{k} \in P^{++}(s)$ (or the corresponding $q$-wedge product $u_{\boldsymbol{k}}$) \
can be represented by a bead-configuration on the abacus with $n \ell$ runners 
and beads put on the positions $k_{1},k_{2},\cdots$.
We call this configuration {\it the abacus presentation} of $u_{\boldsymbol{k}}$. 

\begin{example}
If $n=2$, $\ell =3$, $s=0$, and $\boldsymbol{k} = (6,3,2,1,-2,-4,-5,-7,-8,-9,\cdots)$, 
then the abacus presentation of $u_{\boldsymbol{k}}$ is 

$$\begin{array}{cc|cc|cccc}
d=1 & & d=2 & & d=3 & & & \\ 
\vdots & \vdots & \vdots & \vdots & \vdots & \vdots & & \\
\centercircle{-17} & \centercircle{-16} & \centercircle{-15} & \centercircle{-14} & \centercircle{-13} & \centercircle{-12} 
		& \hspace{1em} & \cdots m=3  \\ 
\centercircle{-11} & \centercircle{-10} & \centercircle{-9} & \centercircle{-8} & \centercircle{-7} & -6 & \hspace{1em} & \cdots m=2  \\ 
\centercircle{-5} & \centercircle{-4} & -3 & \centercircle{-2} & -1 & 0 & \hspace{1em} & \cdots m=1  \\ 
\centercircle{1} & \centercircle{2} & \centercircle{3} & 4 & 5 & \centercircle{6} & \hspace{1em} & \cdots m=0  \\
\vdots & \vdots & \vdots & \vdots & \vdots & \vdots & & \\
c=1 & c= 2 & c=1 & c= 2 & c=1 & c= 2 & &
\end{array}$$
\end{example}

We use another labeling of runners and positions. 
Given an integer $k$, let $c,d$ and $m$ be the unique integers satisfying 
\begin{equation}
k=c+n(d-1)-n \ell m \quad , \quad 1 \le c \le n \quad \text{ and } \quad 1 \le d \le \ell. 
\label{k->cdm}
\end{equation}
Then, in the abacus presentation, the position $k$ is on the $c + n(d-1)$-th runner (see the previous example). 
Relabeling the position $k$ by $c - n m$, we have $\ell$ abaci with $n$ runners.

\begin{example}
In the previous example, relabeling the position $k$ by $c - n m$, we have 

$$\begin{array}{cc|cc|cccc}
d=1 & & d=2 & & d=3 & & & \\ 
\vdots & \vdots & \vdots & \vdots & \vdots & \vdots & & \\
\centercircle{-5} & \centercircle{-4} & \centercircle{-5} & \centercircle{-4} & \centercircle{-5} & \centercircle{-4} 
		& \hspace{1em} & \cdots m=3  \\ 
\centercircle{-3} & \centercircle{-2} & \centercircle{-3} & \centercircle{-2} & \centercircle{-3} & -2 & \hspace{1em} & \cdots m=2  \\ 
\centercircle{-1} & \centercircle{0} & -1 & \centercircle{0} & -1 & 0 & \hspace{1em} & \cdots m=1  \\ 
\centercircle{1} & \centercircle{2} & \centercircle{1} & 2 & 1 & \centercircle{2} & \hspace{1em} & \cdots m=0  \\
\vdots & \vdots & \vdots & \vdots & \vdots & \vdots & & \\
c=1 & c= 2 & c=1 & c= 2 & c=1 & c= 2 & &
\end{array}$$

\label{exam.2.5}
\end{example}

We assign to each of $\ell$ abacus presentations with $n$ runners a $q$-wedge product of level $1$. 
In fact, straightening rules in each ``sector'' are the same as those of level $1$ 
	by identifying the abacus in the sector with that of level $1$. 
(see also \cite{U} and \S 4.1 for the detail)

We introduce some notation.

\begin{definition}
For an integer $k$, let $c,d$ and $m$ be the unique integers satisfying (\ref{k->cdm}), and write 

\begin{equation}
u_{k} = u_{c - n m}^{(d)}.
\end{equation}

Also we write $u_{c_{1} - n m_{1}}^{(d_{1})} > u_{c_{2} - n m_{2}}^{(d_{2})}$ 
if $k_{1} > k_{2}$, where $k_{i} = c_{i} + n (d_{i} - 1) - n \ell m_{i}$, $(i=1,2)$. 
\label{notation}
\end{definition}

We regard $u_{c - n m}^{(d)}$ as $u_{c - n m}$ in the case of $\ell = 1$.

\begin{example}
If $n=2$, $\ell = 3$, then we have 
$$ u_{-10} \wedge u_{1} = -q^{-1} \, u_{1} \wedge u_{-10} +(q^{-2} -1) \, u_{-4} \wedge u_{-5}, $$
that is, 
$$ u_{-2}^{(1)} \wedge u_{1}^{(1)} = -q^{-1} \, u_{1}^{(1)} \wedge u_{-2}^{(1)} +(q^{-2} -1) \, u_{0}^{(1)} \wedge u_{-1}^{(1)}. $$

On the other hand, in the case of $n=2, \ell =1$, 
$$ u_{-2} \wedge u_{1} = -q^{-1} \, u_{1} \wedge u_{-2} +(q^{-2} -1) \, u_{0} \wedge u_{-1}.$$
\end{example}

\subsection{$\ell$-tuples of Young diagrams}

Another indexation of the ordered $q$-wedge products is given by the set of pairs 
$( \boldsymbol{\lambda} , \boldsymbol{s} )$ 
of $\ell$-tuples of Young diagrams $\boldsymbol{\lambda} = ( \lambda^{(1)} , \cdots ,\lambda^{(\ell)})$
and integer sequences $\boldsymbol{s} = (s_{1} , \cdots , s_{\ell})$ summing up to $s$.
Let $\boldsymbol{k} = (k_{1} , k_{2} , \cdots ) \in P^{++}(s)$, and write 
$$k_{r} = c_{r} + n(d_{r} - 1) - n \ell m_{r} \quad , 
	\quad 1 \le c_{r} \le n \quad , \quad 1 \le d_{r} \le \ell \quad , \quad m_{r} \in \mathbb{Z} \quad . $$
For $d \in \{ 1 ,2 , \cdots , \ell \}$, let $k_{1}^{(d)}, k_{2}^{(d)}, \cdots$ be integers such that 

$$\beta^{(d)} = \{ c_{r} - n m_{r} \,\, | \,\, d_{r} = d \} = \{ k_{1}^{(d)}, k_{2}^{(d)}, \cdots \} \quad \text{ and } \quad 
	k_{1}^{(d)} > k_{2}^{(d)} > \cdots $$ 
Then we associate to the sequence $(k_{1}^{(d)} , k_{2}^{(d)} , \cdots )$ an integer $s_{d}$ 
	and a partition $\lambda^{(d)}$ by 
$$ k_{r}^{(d)} = s_{d} -r+1 \quad \text{ for sufficiently large } r 
	\quad \text{ and } \quad  
	\lambda^{(d)}_{r} = k^{(d)}_{r} - s_{d} + r -1 \quad  \text { for } r \ge 1. $$
In this correspondence, we also write 

\begin{equation}
u_{\boldsymbol{k}} = | \boldsymbol{\lambda} ; \boldsymbol{s} \rangle \quad (\boldsymbol{k} \in P^{++}(s)).
\end{equation}

\begin{example}
If $n=2$, $\ell =3$, $s=0$, and $\boldsymbol{k} = (6,3,2,1,-2,-4,-5,-7,-8,-9,\cdots)$, then

\begin{align*}
k_{1} &= 6 = 2 + 2(3-1) -6 \cdot 0 \,\,\,\,,\,\,\,\,
k_{2} = 3 = 1 + 2(2-1) -6 \cdot 0 \,\,\,\,, \\
k_{3} &= 2 = 2 + 2(1-1) -6 \cdot 0 \,\,\,\,\,\, ,  \cdots \text{ and so on.} 
\end{align*}
Hence, 
\begin{align*}
\beta^{(1)} = \{ 2,1,0,-1,-2, \cdots \} \,\,\,\,\, , \,\,\,\,\,
\beta^{(2)} = \{ 1,0,-2,-3,-4, \cdots \}  \,\,\,\,\,\, , \,\,\,\,\,
\beta^{(3)} = \{ 2,-3,-4,-5, \cdots \}  \,\,\,\, .
\end{align*}
Thus, $\boldsymbol{s} = (2,0,-2)$ and $\boldsymbol{\lambda} = (\emptyset , (1,1) , (4))$.

Note that we can read off $\boldsymbol{s} = (2,0,-2)$ and $\boldsymbol{\lambda} = (\emptyset , (1,1) , (4))$ 
from the abacus presentation. 
$($see Example \ref{exam.2.5}$)$
\end{example}

\subsection{The $q$-deformed Fock spaces of higher levels}

\begin{definition}
For $\boldsymbol{s} \in \mathbb{Z}^{\ell}$, we define the $q$-deformed Fock space 
$\boldsymbol{F}_{q}[\boldsymbol{s}]$ of level $\ell$ to be the subspace of $\Lambda^{s}$ spanned by 
$| \boldsymbol{\lambda} ; \boldsymbol{s} \rangle$ $(\boldsymbol{\lambda} \in \Pi^{\ell}) \colon $  

\begin{equation}
\boldsymbol{F}_{q}[\boldsymbol{s}] = 
\bigoplus_{\boldsymbol{\lambda} \in \Pi^{\ell}} \mathbb{Q}(q) \, | \boldsymbol{\lambda} ; \boldsymbol{s} \rangle.
\end{equation}
We call $\boldsymbol{s}$ a {\it multi charge}. 
\end{definition}

\subsection{The bar involution}

\begin{definition}
The involution $\overline{\phantom{xy}}$ of $\Lambda^{s}$ is the $\mathbb{Q}$-vector space automorphism such that 
$\overline{q} = q^{-1}$ and 

\begin{align}
\overline{ u_{\boldsymbol{k}} } 
&= \overline{ u_{k_{1}} \wedge \cdots \wedge u_{k_{r}} } \wedge u_{k_{r+1}} \wedge \cdots 
=(-q)^{\kappa (d_{1}, \cdots ,d_{r})} q^{-\kappa (c_{1},\cdots,c_{r})}  
(u_{k_{r}} \wedge \cdots \wedge u_{k_{1}} )\wedge u_{k_{r+1}} \wedge \cdots ,
\label{barinvolution}
\end{align}
where $c_{i}$, $d_{i}$ are defined by $k_{i}$ as in (\ref{k->cdm}), $r$ is an integer satisfying $k_{r} = s-r+1$. 
And $\kappa (a_{1}, \cdots ,a_{r})$ is defined by  
$$ \kappa (a_{1}, \cdots ,a_{r}) = \# \{ (i,j) \, | \, i<j \,,\, a_{i}=a_{j} \} .$$
\label{barinv}
\end{definition}

\bf{Remarks} \rm{(i)} The involution is well defined. i.e. it doesn't depend on $r$ \cite{U}.

(ii) The involution comes from the bar involution of affine Hecke algebra $\hat{H_{r}}$. 
(see \S 4 for more detail.)

(iii) The involution preserves the $q$-deformed Fock space $\boldsymbol{F}_{q}[\boldsymbol{s}]$ of higher level.

\subsection{The dominance order}

We define a partial ordering 
$ | \boldsymbol{\lambda} ; \boldsymbol{s} \rangle \ge | \boldsymbol{\mu} ; \boldsymbol{s} \rangle $. 
For $ | \boldsymbol{\lambda} ; \boldsymbol{s} \rangle $ and 
$ | \boldsymbol{\mu} ; \boldsymbol{s} \rangle $, 
we define multi-sets $\widetilde{\boldsymbol{\lambda}}$ and $\widetilde{\boldsymbol{\mu}}$ as 

\begin{align*}
\widetilde{\boldsymbol{\lambda}} 
&= \{ \lambda_{a}^{(d)} + s_{d}  \, | \, 1 \le d \le \ell  \, , \, 1 \le a \le \max ( l(\lambda^{(d)}) , l(\mu^{(d)}) )  \, \}  \,, \\
\widetilde{\boldsymbol{\mu}} 
&= \{ \mu_{a}^{(d)} + s_{d}  \, | \, 1 \le d \le \ell  \, , \, 1 \le a \le \max ( l(\lambda^{(d)}) , l(\mu^{(d)}) )  \, \} \,  .
\end{align*}
We denote by $( \tilde{\lambda}_{1} , \tilde{\lambda}_{2} , \cdots )$ 
	$($resp. $(\tilde{\lambda}_{1} , \tilde{\lambda}_{2} , \cdots ))$ the sequence obtained by rearranging the elements 
	in the multi-set $\widetilde{\boldsymbol{\lambda}}$ (resp. $\widetilde{\boldsymbol{\mu}} )$ in decreasing order.  

\begin{definition}
Let $ | \boldsymbol{\lambda} ; \boldsymbol{s} \rangle  = u_{k_{1}} \wedge u_{k_{2}} \wedge \cdots \,\,$ and \,\,
$ | \boldsymbol{\mu} ; \boldsymbol{s} \rangle = u_{g_{1}} \wedge u_{g_{2}} \wedge \cdots  $. 
We define $ | \boldsymbol{\lambda} ; \boldsymbol{s} \rangle \ge | \boldsymbol{\mu} ; \boldsymbol{s} \rangle $ 
if $| \boldsymbol{\lambda} | = | \boldsymbol{\mu} |$ and 

\begin{equation}
\begin{cases}
\mathrm{(a)} \quad \quad 
\widetilde{ \boldsymbol{\lambda} } \not= \widetilde{ \boldsymbol{\mu} } \quad , \quad 
\sum_{j=1}^{r} \tilde{\lambda}_{j} \ge \sum_{j=1}^{r} \tilde{\mu}_{j} \,\,\,\,\,\, (\text{for all } \,\,\, r=1,2,3,\cdots ) 
	\quad , \text{ or } \quad  \\
\mathrm{(b)} \quad \quad 
\widetilde{ \boldsymbol{\lambda} } = \widetilde{ \boldsymbol{\mu} } \quad , \quad 
\sum_{j=1}^{r} k_{j} \ge \sum_{j=1}^{r} g_{j} \,\,\,\,\,\, (\text{for all } \,\,\, r=1,2,3,\cdots ) \quad .
\end{cases} 
\end{equation}
\label{order}
\end{definition}

\bf{Remark}. \rm{} The order in Definition \ref{order} is different from the order in \cite{U} (see Example \ref{example1} below). 
However, the unitriangularity in (\ref{unitriangularity}) holds for both of them. 

\begin{example}
Let $n= \ell =2$, $\boldsymbol{s} = (1,-1)$, $\boldsymbol{\lambda} = ((1,1) , \emptyset)$, and
$\boldsymbol{\mu} = (\emptyset , (2))$. 
Then, $| \boldsymbol{\lambda} ; \boldsymbol{s} \rangle = u_{2} \wedge u_{1} \wedge u_{-1} \wedge u_{-3} \wedge \cdots$ and 
$| \boldsymbol{\mu} ; \boldsymbol{s} \rangle = u_{3} \wedge u_{1} \wedge u_{-2} \wedge u_{-3} \wedge \cdots$.
In Uglov's order, $| \boldsymbol{\mu} ; \boldsymbol{s} \rangle$ is greater than 
	$| \boldsymbol{\lambda} ; \boldsymbol{s} \rangle$. 
However, 
$| \boldsymbol{\lambda} ; \boldsymbol{s} \rangle > | \boldsymbol{\mu} ; \boldsymbol{s} \rangle$ under our order 
since $ \{ \tilde{\lambda}_{1} , \tilde{\lambda}_{2} , \tilde{\lambda}_{3} \} = \{ 2,2,-1 \} $ and  
$\{ \tilde{\mu}_{1} , \tilde{\mu}_{2} , \tilde{\mu}_{3} \} = \{ 1,1,1 \}$.
\label{example1}
\end{example}

We define a matrix $(a_{\boldsymbol{\lambda},\boldsymbol{\mu}}(q))_{\boldsymbol{\lambda},\boldsymbol{\mu}}$ by 

\begin{equation}
\overline{ |\boldsymbol{\lambda} ; \boldsymbol{s} \rangle } = 
\sum_{\boldsymbol{\mu}} a_{\boldsymbol{\lambda},\boldsymbol{\mu}}(q) \,
|\boldsymbol{\mu} ; \boldsymbol{s} \rangle .
\end{equation}
Then the matrix $(a_{\boldsymbol{\lambda},\boldsymbol{\mu}}(q))_{\boldsymbol{\lambda},\boldsymbol{\mu}}$
is unitriangular with respect to $\ge$, that is 
\begin{equation}
\begin{cases}
\mathrm{(a)} & 
\text{ if } \,\, a_{\boldsymbol{\lambda},\boldsymbol{\mu}}(q) \not= 0 \,\, \text{, then } \,\,
|\boldsymbol{\lambda} ; \boldsymbol{s} \rangle \ge |\boldsymbol{\mu} ; \boldsymbol{s} \rangle , \\
\mathrm{(b)} & 
a_{\boldsymbol{\lambda},\boldsymbol{\lambda}}(q) = 1 .
\label{unitriangularity}
\end{cases}
\end{equation}
(see the identity (\ref{eq.22}) for the detail.)

Thus, by the standard argument, the unitriangularity implies the following theorem.

\begin{theorem}\cite{U}
There exist unique bases 
$\{ G^{+}(\boldsymbol{\lambda} ; \boldsymbol{s}) \, | \, \boldsymbol{\lambda} \in \Pi^{\ell}  \}$ and 
$\{ G^{-}(\boldsymbol{\lambda} ; \boldsymbol{s}) \, | \, \boldsymbol{\lambda} \in \Pi^{\ell}  \}$ 
of $\boldsymbol{F}_{q}[\boldsymbol{s}]$ such that 

\begin{align*}
\mathrm{(i) } \hspace{5em}
\overline{G^{+}(\boldsymbol{\lambda} ; \boldsymbol{s})} = G^{+}(\boldsymbol{\lambda} ; \boldsymbol{s}) 
\hspace{2em} , \hspace{3em} & 
\overline{G^{-}(\boldsymbol{\lambda} ; \boldsymbol{s})} = G^{-}(\boldsymbol{\lambda} ; \boldsymbol{s}) \\
\mathrm{(ii) } \hspace{1em}
G^{+}(\boldsymbol{\lambda} ; \boldsymbol{s}) \equiv | \, \boldsymbol{\lambda} ; \boldsymbol{s} \rangle
\,\,\,\, \mathrm{mod} \,\, q \, \mathcal{L}^{+} 
\hspace{2em} , \hspace{3em} & 
G^{-}(\boldsymbol{\lambda} ; \boldsymbol{s}) \equiv | \, \boldsymbol{\lambda} ; \boldsymbol{s} \rangle
\,\,\,\, \mathrm{mod} \,\, q^{-1} \, \mathcal{L}^{-}   \\
\text{where} \hspace{5em}
\mathcal{L}^{+} = \bigoplus_{\boldsymbol{\lambda} \in \Pi^{\ell}} 
\mathbb{Q}[q] \, | \boldsymbol{\lambda} ; \boldsymbol{s} \rangle
\hspace{2em} , \hspace{3em} & 
\mathcal{L}^{-} = \bigoplus_{\boldsymbol{\lambda} \in \Pi^{\ell}} 
\mathbb{Q}[q^{-1}] \, | \boldsymbol{\lambda} ; \boldsymbol{s} \rangle .
\end{align*}
\end{theorem}

\begin{definition}
Define matrices 
$\Delta^{+}(q) = (\Delta^{+}_{\boldsymbol{\lambda},\boldsymbol{\mu}}(q))
_{\boldsymbol{\lambda},\boldsymbol{\mu}}$ and 
$\Delta^{-}(q) = (\Delta^{-}_{\boldsymbol{\lambda},\boldsymbol{\mu}}(q))
_{\boldsymbol{\lambda},\boldsymbol{\mu}}$ by 

\begin{align}
G^{+}(\boldsymbol{\lambda} ; \boldsymbol{s}) = \sum_{\boldsymbol{\mu}} 
\Delta^{+}_{\boldsymbol{\lambda},\boldsymbol{\mu}}(q) \, | \, \boldsymbol{\mu} ; \boldsymbol{s} \rangle 
\hspace{2em} , \hspace{3em} 
G^{-}(\boldsymbol{\lambda} ; \boldsymbol{s}) = \sum_{\boldsymbol{\mu}} 
\Delta^{-}_{\boldsymbol{\lambda},\boldsymbol{\mu}}(q) \, | \, \boldsymbol{\mu} ; \boldsymbol{s} \rangle .
\end{align}
\end{definition}

The entries $\Delta^{\pm}_{\boldsymbol{\lambda},\boldsymbol{\mu}}(q)$ are called {\it $q$-decomposition numbers}. 
Note that $q$-decomposition numbers $\Delta^{\pm}(q)$ depend on $n$, $\ell$ and $\boldsymbol{s}$.
The matrices $\Delta^{+}(q)$ and $\Delta^{-}(q)$ are also unitriangular with respect to $\ge$.

It is known \cite[Theorem 3.26]{U} that the entries of $\Delta^{-}(q)$ are 
	Kazhdan-Lusztig polynomials of parabolic submodules of affine Hecke algebras of type $A$, 
	and that they are polynomials in $q$ with non-negative integer coefficients (see \cite{KT}).

%
%

\section{A comparison of $q$-decomposition numbers}

\subsection{Sufficiently large and sufficiently small}

\begin{definition}
\rm{}Let $\boldsymbol{s} = (s_{1}, s_{2}, \cdots , s_{\ell}) \in \mathbb{Z}^{\ell}$ be a multi charge and $1 \leq j \leq \ell$.

(i). We say that the $j$-th component $s_j$ of the multi charge $\boldsymbol{s}$ is {\it sufficiently large} for 
$|\boldsymbol{\lambda} ; \boldsymbol{s} \rangle \in \boldsymbol{F}_{q}[\boldsymbol{s}]$ if 
\begin{equation}
s_{j} - s_{i} \geq \lambda^{(i)}_{1} 
\hspace{2em} \text{for all } \,\,\, i=1,2,\cdots , \ell .
\end{equation}
More generally, we say that $s_{j}$ is sufficiently large for a $q$-wedge $u_{\boldsymbol{k}}$ if  

\begin{equation}
s_{j} \ge c_{r} - n m_{r}  \hspace{2em} \text{ for all } r=1,2,\cdots ,
\end{equation}
where $k_{r} = c_{r} + n (d_{r} - 1) - n \ell m_{r}$, $(r = 1,2, \cdots)$, $1 \le c \le n$ and $1 \le d \le \ell$ (see \S 2).

(ii). We say that $s_{j}$ is {\it sufficiently small} for $| \boldsymbol{\lambda} ; \boldsymbol{s} \rangle $ if 
\begin{equation}
s_{i} - s_{j} \geq | \boldsymbol{\lambda} | = |\lambda^{(1)}| + \cdots + |\lambda^{(\ell)}| 
\hspace{2em} \text{for all } \,\,\, i \not= j .
\end{equation}
\label{largesmall}
\end{definition}

Note that the definition of sufficiently small depends only on the size of $\boldsymbol{\lambda}$ 
	and the multi charge $\boldsymbol{s}$. 
When we fix the multi charge $\boldsymbol{s}$, we say that $s_{j}$ is {\it sufficiently small} for $N$ if 
\begin{equation}
s_{i} - s_{j} \geq N \hspace{2em} \text{for all } \,\,\, i \not= j .
\end{equation}

\bf{Remark}. \rm{} If $|\boldsymbol{\lambda} ; \boldsymbol{s} \rangle$ is $0$-dominant in the sense of \cite{U}, that is 
$$ s_{i} - s_{i+1} \geq |\boldsymbol{\lambda} | = |\lambda^{(1)}| + \cdots + |\lambda^{(\ell)}| 
\hspace{2em} \text{for all } \,\,\, i=1,2,\cdots, \ell -1 \,\,, $$
then $s_{1}$ is sufficiently large for $|\boldsymbol{\lambda} ; \boldsymbol{s} \rangle$ 
and $s_{\ell}$ is sufficiently small for $|\boldsymbol{\lambda} ; \boldsymbol{s} \rangle$.

\begin{lemma}
If $s_{j}$ is sufficiently large for $|\boldsymbol{\lambda} ; \boldsymbol{s} \rangle$ and 
$|\boldsymbol{\lambda} ; \boldsymbol{s} \rangle \ge |\boldsymbol{\mu} ; \boldsymbol{s} \rangle$, 
then

$\mathrm{(i)}$ $\lambda^{(j)} = \emptyset$,

$\mathrm{(ii)}$ $s_{j}$ is also sufficiently large for $|\boldsymbol{\mu} ; \boldsymbol{s} \rangle$. 
In particular, $\mu^{(j)} = \emptyset$.
\label{lem.3.2}
\end{lemma}

\begin{proof}
It is clear that $\lambda^{(j)} = \emptyset$ by the definition.

Note that

\begin{align*}
s_{j} \text{ is sufficiently large for } |\boldsymbol{\lambda} ; \boldsymbol{s} \rangle  
& \Leftrightarrow s_{j} - s_{i} \ge \lambda^{(i)}_{1} \hspace{3em} \text{ for all } i=1,2,\cdots,\ell \\
& \Leftrightarrow s_{j} \ge  \max \{\lambda_{1}^{(1)} + s_{1} , \cdots , \lambda_{1}^{(\ell)} + s_{\ell} \} = \tilde{\lambda}_{1} .
\end{align*}

If $|\boldsymbol{\lambda} ; \boldsymbol{s} \rangle \ge |\boldsymbol{\mu} ; \boldsymbol{s} \rangle$, 
	then $\tilde{\lambda}_{1} \ge \tilde{\mu}_{1}$ and so $s_{j} \ge \tilde{\mu}_{1}$. 
It means that $s_{j}$ is sufficiently large for $|\boldsymbol{\mu} ; \boldsymbol{s} \rangle$.
\end{proof}

\begin{lemma}
Suppose that $s_{j}$ is sufficiently small for $|\boldsymbol{\lambda} ; \boldsymbol{s} \rangle$. 
If $|\boldsymbol{\lambda} ; \boldsymbol{s} \rangle \ge |\boldsymbol{\mu} ; \boldsymbol{s} \rangle$ and 
$\mu^{(j)} = \emptyset$, 
then $\lambda^{(j)} = \emptyset$.
\label{lem9}
\end{lemma}

\begin{proof}
Suppose that $l( \lambda^{(j)} ) \ge 1$.
Then $s_{j}$ is the minimal integer in the set 
	$\{ \mu_{a}^{(d)} + s_{d}  \, | \, 1 \le d \le \ell  \, , \, 1 \le a \le \max ( l(\lambda^{(d)}) , l(\mu^{(d)}) )  \, \} \}$ 
	because $\mu^{(j)} = \emptyset$ and $s_{j}$ is the minimal integer in $\boldsymbol{s}$. 
On the other hand, the minimal integer in the set 
	$\{ \lambda_{a}^{(d)} + s_{d}  \, | \, 1 \le d \le \ell  \, , \, 1 \le a \le \max ( l(\lambda^{(d)}) , l(\mu^{(d)}) )  \, \} \}$ 
	is greater than $s_{j}$ because $s_{j}$ is sufficiently small for $|\boldsymbol{\lambda} ; \boldsymbol{s} \rangle$.
Therefore 
$|\boldsymbol{\lambda} ; \boldsymbol{s} \rangle \not\ge |\boldsymbol{\mu} ; \boldsymbol{s} \rangle$.
This is a contradiction.

\end{proof}

\subsection{Main results}

Now, we are ready to state our main theorems. 
We will prove the theorems in \S5 and \S6 respectively.

\begin{theorem}
Let $\varepsilon \in \{ +,- \}$. 
If $s_{j}$ is sufficiently large for $|\boldsymbol{\lambda} ; \boldsymbol{s} \rangle$, then 

\begin{equation}
\Delta^{\varepsilon}_{ \boldsymbol{\lambda} , \boldsymbol{\mu} ; \boldsymbol{s} }(q)
=\Delta^{\varepsilon}_{ \check{\boldsymbol{\lambda}} , \check{\boldsymbol{\mu}} ; \check{\boldsymbol{s}} }(q),
\end{equation}
where $\check{\boldsymbol{\lambda}}$ (resp. $\check{\boldsymbol{\mu}} , \check{\boldsymbol{s}})$ 
is obtained by omitting the $j$-th component of $\boldsymbol{\lambda}$ 
(resp. $\boldsymbol{\mu} , \boldsymbol{s})$.
\label{thmA}
\end{theorem}

\begin{theorem}
Let $\varepsilon \in \{ +,- \}$. 
If $s_{j}$ is sufficiently small for $|\boldsymbol{\mu} ; \boldsymbol{s} \rangle$ and $\mu^{(j)} = \emptyset$, then 

\begin{equation}
\Delta^{\varepsilon}_{ \boldsymbol{\lambda} , \boldsymbol{\mu} ; \boldsymbol{s} }(q)
=\Delta^{\varepsilon}_{ \check{\boldsymbol{\lambda}} , \check{\boldsymbol{\mu}} ; \check{\boldsymbol{s}} }(q),
\end{equation}
where $\check{\boldsymbol{\lambda}}$ (resp. $\check{\boldsymbol{\mu}} , \check{\boldsymbol{s}})$ 
is obtained by omitting the $j$-th component of $\boldsymbol{\lambda}$ 
(resp. $\boldsymbol{\mu} , \boldsymbol{s})$.
\label{thmB}
\end{theorem}

\begin{example}
$\mathrm{(i)}$ 
If $n = \ell = 2 $, $\boldsymbol{s} = (3,-3)$ and 
$\boldsymbol{\lambda} = (\emptyset , (6))$, 
$\boldsymbol{\mu} = (\emptyset , (5 , 1))$, 
then $s_{1}$ is sufficiently large for $|\boldsymbol{\lambda} ; \boldsymbol{s} \rangle$.
Hence 

\begin{align*}
\Delta^{-}_{ \boldsymbol{\lambda} , \boldsymbol{\mu} ; \boldsymbol{s} }(q)
&=\Delta^{-}_{ \check{\boldsymbol{\lambda}} , \check{\boldsymbol{\mu}} ; \check{\boldsymbol{s}} }(q) 
=\Delta^{-}_{ (6) , (5,1) ; (-3) }(q) 
= -q^{-1}.
\end{align*}

$\mathrm{(ii)}$ 
If $n = \ell = 2 $, $\boldsymbol{s} = (3,-3)$ and 
$\boldsymbol{\lambda} = ((6) , \emptyset)$, 
$\boldsymbol{\mu} = ((5 , 1) , \emptyset)$, 
then $s_{2}$ is sufficiently small for $|\boldsymbol{\mu} ; \boldsymbol{s} \rangle$. 
Hence 

\begin{align*}
\Delta^{-}_{ \boldsymbol{\lambda} , \boldsymbol{\mu} ; \boldsymbol{s} }(q)
&=\Delta^{-}_{ \check{\boldsymbol{\lambda}} , \check{\boldsymbol{\mu}} ; \check{\boldsymbol{s}} }(q) 
=\Delta^{-}_{ (6) , (5,1) ; (-3) }(q) 
= -q^{-1}.
\end{align*}

\end{example}

%
%

\section{ q-wedges and straightening rules}

In this section, we review the straightening rules \cite{U} to prove our main results.

\subsection{affine Hecke algebra and straightening rules}

In this paragraph, we review the affine Hecke algebra of type $A_{1}$ and straightening rules. 
We treat only the case of type $A_{1}$. 
Indeed for our proof we only need the straightening rule of $q$-wedge whose length is equal to two. 
In fact, the straightening rules for a $q$-wedge which length is greater than $2$ is obtained from 
the straightening rules for two adjacent element. (see Example \ref{ex3})
More general case, see \cite{U}.

The Hecke algebra $H$ of type $A_{1}$ is the algebra over $\mathbb{Q}(q)$ with generator $T_{1}$ and relation

\begin{equation}
(T_{1} - q^{-1}) (T_{1} + q) = 0 . 
\end{equation}

The affine Hecke algebra $\hat{H}$ is the tensor space $H$ 
and the polynomial ring $\mathbb{Q}(q)[X_{1}^{\pm},X_{2}^{\pm}]$ with relations 

\begin{align}
X^{\lambda} T_{1} &= T_{1} X^{s_{1}(\lambda)} + (q - q^{-1}) \frac{X^{s_{1}(\lambda)} - X^{\lambda}}{ 1 - X_{1}X_{2}^{-1} } \,\,\, , 
\label{relation1}  \\
T_{1} X^{\lambda} &= X^{s_{1}(\lambda)} T_{1} + (q - q^{-1}) \frac{X^{s_{1}(\lambda)} - X^{\lambda}}{ 1 - X_{1}X_{2}^{-1} } \,\,\, ,
\end{align}

where $\lambda \in \mathbb{Z}^{2}$ and $s_{1}$ is the transposition.

Let $P_{1}$ be the $\mathbb{Q}(q)$-vector space whose basis is 
$\{ \,\, ( c_{1} , c_{2} | \,\, | \,\, c_{1} , c_{2} \in \mathbb{Z} , 1 \le c_{1} \le n , 1 \le c_{2} \le n \}$. 
Define the right action of $H$ on $P_{1}$ as  

\begin{align}
( c_{1} , c_{2} | \cdot T_{1} =
\begin{cases}
\,\, ( c_{2} , c_{1} | & \text{ if } c_{1} < c_{2}  \,\,, \\
\,\, q^{-1} \, ( c_{1} , c_{1} | & \text{ if } c_{1} = c_{2} \,\,, \\
\,\, ( c_{2} , c_{1} | - (q-q^{-1}) \, ( c_{1} , c_{2} |  & \text{ if } c_{1} > c_{2}  \,\,.
\end{cases}
\label{rightaction}
\end{align}

Let $P_{2}$ be the $\mathbb{Q}(q)$-vector space whose basis is 
$\{ \,\, | d_{1},d_{2} ) \,\, | \,\,  d_{1} , d_{2} \in \mathbb{Z} , 1 \le d_{1} \le \ell , 1 \le d_{2} \le \ell \}$. 
Define the left action of $H$ on $P_{2}$ as  

\begin{align}
T_{1} \cdot | d_{1} , d_{2} ) =
\begin{cases}
\,\, | d_{2} , d_{1} ) & \text{ if } d_{1} < d_{2}  \,\,, \\
\,\, -q \, | d_{1} , d_{1} ) & \text{ if } d_{1} = d_{2} \,\,, \\
\,\, | d_{2} ,  d_{1} ) - (q-q^{-1}) \, | d_{1} , d_{2} )  & \text{ if } d_{1} > d_{2}  \,\,.
\end{cases} 
\label{leftaction}
\end{align}

Define a vector space $\Lambda$ by 

\begin{equation*}
\Lambda = P_{1} \otimes_{H} \hat{H} \otimes_{H} P_{2} . 
\end{equation*}

\begin{definition}\cite{U}
For $( c_{1} , c_{2} | \in P_{1}$, $ | d_{1} , d_{2} ) \in P_{2}$, and $m_{1},m_{2} \in \mathbb{Z}$, 
put $k_{j} = c_{j} + n (d_{j} - 1) - n \ell m_{j}$, $(j=1,2)$.
Denote $( c_{1} , c_{2} | \otimes X_{1}^{m_{1}} X_{2}^{m_{2}} \otimes | d_{1} , d_{2} ) \in \Lambda$ by 

\begin{equation}
u_{k_{1}} \wedge u_{k_{2}} .
\end{equation}
\label{q-wedge}
\end{definition}

\begin{proposition}\cite{U}
For integers $k_{1},k_{2}$, let $c_{j},d_{j},m_{j}$ be the unique integers satisfying $k_{j} = c_{j} + n (d_{j} - 1) - n \ell m_{j}$, 
$1 \le c_{j} \le n$ and $1 \le d_{j} \le \ell$, $(j=1,2)$.
Then, 

\begin{align}
u_{k_{2}} \wedge u_{k_{1}} = 
(-q^{-1})^{\delta_{d_{1} = d_{2} }} \biggl\{
q^{\alpha} \, u_{k_{1}} \wedge u_{k_{2}} + \mathrm{sgn}(m) \,
(q - q^{-1}) \sum_{j= \beta }^{ | m_{1} - m_{2} |  - \gamma} 
u_{k_{1} - c_{1}+c_{2} - \mathrm{sgn}(m) n \ell j} \wedge u_{k_{2} + c_{1} - c_{2} + \mathrm{sgn} (m) n \ell j}
\biggr\},
\label{eq.21}
\end{align}
where 
$$\mathrm{sgn}(m) = 
\begin{cases}
1 & \text{ if } \,\, m_{1} < m_{2} \\
-1 & \text{ if } \,\, m_{1} > m_{2} \\
0 & \text{ if } \,\, m_{1} = m_{2} 
\end{cases}  \quad , \quad    
\alpha =
\begin{cases}
1 & \text{ if } \,\, c_{1} = c_{2} \text{ and } k_{1} > k_{2} \\
-1 & \text{ if } \,\, c_{1} = c_{2} \text{ and } k_{1} < k_{2} \\
0 & \text{ if } \,\, c_{1} \not= c_{2}
\end{cases} \quad , $$ 
$$ \delta_{d_{1} = d_{2}}  =
\begin{cases}
1 & \text{ if } d_{1} = d_{2} \\
0 & \text{ if } d_{1} \not= d_{2}
\end{cases} \quad , \quad 
\beta =
\begin{cases}
0 & \text{ if } \,\, c_{1} > c_{2} , m_{1} < m_{2}  \text{ or } c_{1} < c_{2} , m_{1} > m_{2}   \\
1 & \text{ if } \,\, otherwise
\end{cases} \quad , $$ 
and 
$$ \gamma =
\begin{cases}
1 & \text{ if } \,\, d_{1} < d_{2} , m_{1} < m_{2}  \text{ or } d_{1} > d_{2} , m_{1} > m_{2}   \\
0 & \text{ if } \,\, d_{1} > d_{2} , m_{1} < m_{2}  \text{ or } d_{1} < d_{2} , m_{1} > m_{2}   
\end{cases}. $$
\label{s.rule}
\end{proposition}

\begin{proof}
We only show the statement in the case $c_{1}=c_{2} \,,\, m_{1} < m_{2} \text{, and } d_{1} < d_{2}$.
The other case can be treated similarly.

In this case, $k_{1} > k_{2}$, $\delta_{d_{1} = d_{2}} = 0$, $\mathrm{sgn}(m) = 1$, $\alpha = 1$, $\beta = 1$, and $\gamma = 1$. 
Note that $X_{1}X_{2}$ and $T_{1}$ commute each other thanks to the relation (\ref{relation1}),
that is $X_{1}X_{2} T_{1} = T_{1} X_{1}X_{2}$.
From the relation (\ref{relation1}), for any positive integer $N$ we have

\begin{equation}
X_{1}^{N} T_{1} = T_{1}^{-1} X_{2}^{N} + (q - q^{-1}) ( X_{1} X_{2}^{N-1} + X_{1}^{2} X_{2}^{N-2} + \cdots + X_{1}^{N-1}X_{2} ).
\label{eq.1}
\end{equation}

Hence 

\begin{align*}
& u_{k_{2}} \wedge u_{k_{1}} \\
=& ( c_{1},c_{1} | \otimes X_{1}^{m_{2}} X_{2}^{m_{1}} \otimes | d_{2} , d_{1} ) 
\hspace{5em} \text{ (by Definition \ref{q-wedge}) } \\
=& ( c_{1},c_{1} | \otimes (X_{1}X_{2})^{m_{1}} X_{1}^{m_{2}-m_{1}} T_{1} \otimes | d_{1} , d_{2} )  
\hspace{5em} \text{ (by (\ref{leftaction})) } \\
=& ( c_{1},c_{1} | \otimes (X_{1}X_{2})^{m_{1}} \biggl\{
T_{1}^{-1} X_{2}^{m_{2}-m_{1}} + (q - q^{-1}) ( X_{1} X_{2}^{m_{2}-m_{1}-1} + X_{1}^{2} X_{2}^{m_{2}-m_{1}-2} + 
\cdots + X_{1}^{m_{2}-m_{1}-1}X_{2} ) 
\biggr\} \otimes | d_{1} , d_{2} )  \\
& \hspace{7em} \text{ (by (\ref{eq.1})) } \\
=& q \, ( c_{1},c_{1} | \otimes X_{1}^{m_{1}}X_{2}^{m_{2}} \otimes | d_{1} , d_{2} ) 
+ (q-q^{-1}) \, ( c_{1},c_{1} | \otimes \bigg( X_{1}^{m_{1}+1} X_{2}^{m_{2}-1} + X_{1}^{m_{1}+2} X_{2}^{m_{2}-2} + 
\cdots + X_{1}^{m_{2}-1}X_{2}^{m_{1}+1} \bigg) \otimes | d_{1} , d_{2} ) \\
& \hspace{7em} \text{ (by (\ref{rightaction})) } \\
=& q \, u_{k_{1}} \wedge u_{k_{2}} + (q - q^{-1}) \, \bigg( 
u_{k_{1} - n \ell} \wedge u_{k_{2} + n \ell} + u_{k_{1} - 2 n \ell} \wedge u_{k_{2} + 2 n \ell} + \cdots +
u_{k_{1} - n \ell (m_{2} - m_{1} -1)} \wedge u_{k_{2} + n \ell (m_{2} - m_{1} -1)}  \bigg) \\
& \hspace{7em} \text{ (by Definition \ref{q-wedge}) } \\
=& q \, u_{k_{1}} \wedge u_{k_{2}} + (q - q^{-1}) \, 
\sum_{j=1}^{m_{2} - m_{1} - 1}
u_{k_{1} - n \ell j} \wedge u_{k_{2} + n \ell j}.
\end{align*}

\end{proof}

The identity (\ref{eq.21}) is rewritten in terms of the notation of Definition \ref{notation} as follows. 

\begin{corollary}
Under the same notations in Proposition \ref{s.rule}, we have 
\begin{align}
u_{c_{2} - n m_{2}}^{(d_{2})} \wedge u_{c_{1} - n m_{1}}^{(d_{1})}  
	&= (-q^{-1})^{\delta_{d_{1} = d_{2} }} \,
	q^{\alpha} \, u_{c_{1} - n m_{1}}^{(d_{1})} \wedge u_{c_{2} - c m_{2}}^{(d_{2})}  
	\nonumber \\
	&+ \mathrm{sgn}(m) \,
	(-q^{-1})^{\delta_{d_{1} = d_{2} }} \, (q - q^{-1}) \sum_{j= \beta }^{ | m_{1} - m_{2} |  - \gamma} 
	u_{c_{2} - n m_{1} - \mathrm{sgn}(m) n j}^{(d_{1})} \wedge u_{c_{1} - n m_{2} + \mathrm{sgn} (m) n j}^{(d_{2})}.
\label{eq.22}
\end{align}
\label{s-rules-cor}
\end{corollary}

\bf{Remarks.} \rm{}
(i) Note that the identity (\ref{eq.22}) depends only on the inequality relationship between $d_{1}$ and $d_{2}$
($c_{1}$ and $c_{2}$). 
It is independent of $\ell$. 

Let $\ell$ and $j$ be integers such that $1 \le j \le \ell+1$. 
Let 
$$u = u_{k_{1}}^{(d_{1})} \wedge u_{k_{2}}^{(d_{2})} \wedge \cdots$$ 
be a $q$-wedge product of level $\ell$. 
Define $d_{1}', d_{2}' ,\cdots$ as 
$$d_{r}'=\begin{cases}
	d_{r} & \text{ if } \,\, d_{r}<j \\
	d_{r}+1 & \text{ if } \,\, d_{r} \ge j
	\end{cases} \quad ,  \quad (r=1,2,\cdots) . $$
Then, 
$$u' = u_{k_{1}}^{(d_{1}')} \wedge u_{k_{2}}^{(d_{2}')} \wedge \cdots$$ 
is the $q$-wedge product of level $\ell +1$. 
In this way, we regard a $q$-wedge product $u$ of level $\ell$ as the $q$-wedge product of level $\ell +1$.

(ii). Let $k_{1}' = k_{1} - c_{1}+c_{2} - \mathrm{sgn}(m) n \ell j$ and $k_{2}' = k_{2} + c_{1} - c_{2} + \mathrm{sgn} (m) n \ell j$. 
That is to say, $u_{k_{1}'} \wedge u_{k_{2}'}$ appears in the summation of (\ref{eq.21}).
Then, $k_{1}'$ and $k_{2}'$ satisfy following properties. 

(a). $k_{1}'$ and $k_{2}'$ are in between $k_{1}$ and $k_{2}$, i.e. 
$ k_{1} < k_{i}' < k_{2}$, $(i=1,2)$ or $ k_{1} > k_{i}' > k_{2}$, $(i=1,2)$. 

(b). $k_{1}'$ and $k_{2}'$ swap the \it{c-part} \rm{} with $k_{1}$ and $k_{2}$.
That is, there exist $m_{1}' , m_{2}' \in \mathbb{Z}$ such that 
$k_{1}' = c_{2} + n (d_{1} -1) - n \ell m_{1}'$ and $k_{2}' = c_{1} + n (d_{2} -1) - n \ell m_{2}'$.

(c). $k_{1}' + k_{2}' = k_{1} + k_{2}$.

In abacus presentation, the positions of $k_{1}' , k_{2}'$ and $k_{1} , k_{2}$ look like 

$$\begin{array}{cc|cc}
d=d_{1} & & d=d_{2} &   \\ 
\vdots & \vdots & \vdots & \centercircle{$k_{2}$}    \\
\vdots & \centercircle{$k_{2}'$} & \vdots & \vdots      \\ 
\vdots & \vdots & \vdots & \vdots      \\ 
\vdots & \vdots & \centercircle{$k_{1}'$} & \vdots     \\ 
\centercircle{$k_{1}$} & \vdots & \vdots & \vdots     \\
\vdots & \vdots & \vdots & \vdots  \\
c=c_{1} & c=c_{2} & c=c_{1} & c=c_{2}  
\end{array}$$

\subsection{Several properties of $q$-wedge products}

In this paragraph, we summarize other properties of $q$-wedge products 
	which will be needed in the proof of our main theorems.  

\begin{lemma}[\cite{U}]
If $k \ge t$, then 

$(i)$. $u_{t} \wedge u_{k} \wedge u_{k-1} \wedge \cdots \wedge u_{t} = 0$, 

$(ii)$. $u_{k} \wedge u_{k-1} \wedge \cdots \wedge u_{t} \wedge u_{k} = 0$.
\label{vanishinglemma}
\end{lemma}

More generally, we have

\begin{corollary}
If $k \ge m \ge t$, then

$(i)$. $u_{m} \wedge u_{k} \wedge u_{k-1} \wedge \cdots \wedge u_{t} = 0$, 

$(ii)$. $u_{k} \wedge u_{k-1} \wedge \cdots \wedge u_{t} \wedge u_{m} = 0$.

\label{cor.5}
\end{corollary}

\begin{proof}
The first assertion immediately follows from Lemma \ref{vanishinglemma} (i). 
We prove (ii) by induction on $m-t$.
If $m=t$, then the assertion follows from Lemma \ref{vanishinglemma} (ii). 

Let $m-t >0$. 
From the identity (\ref{eq.21}), we know that there exist $b_{0}(q), \cdots , b_{m-t}(q)$ such that 

$$ u_{t} \wedge u_{m} = \sum_{j=0}^{m-t} b_{j}(q) \, u_{m-j} \wedge u_{t+j} .$$
Then, 

\begin{align*}
u_{k} \wedge \cdots \wedge u_{t+1} \wedge u_{t} \wedge u_{m} 
&= \sum_{j=0}^{m-t} b_{j}(q) \, u_{k} \wedge \cdots \wedge u_{t+1} \wedge u_{m-j} \wedge u_{t+j} 
\end{align*}
Here, by the induction hypothesis, 
$ u_{k} \wedge \cdots \wedge u_{t+1} \wedge u_{m-j} =0$ for all  $ 0 \le j \le m-t$.
Therefore $u_{k} \wedge \cdots \wedge u_{t+1} \wedge u_{t} \wedge u_{m} = 0$.
\end{proof}

The next corollary follows from the above corollary and Corollary \ref{s-rules-cor}. 

\begin{corollary}
If $k \ge m \ge t$ and $1 \le j \le \ell$ , then 

$(i)$. $u_{m}^{(j)} \wedge u_{k}^{(j)} \wedge u_{k-1}^{(j)} \wedge \cdots \wedge u_{t}^{(j)} = 0$, 

$(ii)$. $u_{k}^{(j)} \wedge u_{k-1}^{(j)} \wedge \cdots \wedge u_{t}^{(j)} \wedge u_{m}^{(j)}= 0$.

\end{corollary}

\begin{definition}
Let

\begin{align*}
u &= u_{k_{1}}^{(d_{1})} \wedge u_{k_{2}}^{(d_{2})} \wedge \cdots \wedge u_{k_{r}}^{(d_{r})} \,\,\,\, , \,\,\,\,
k_{a} = c_{a} - n m_{a} \,\,\,\,,\,\,\, (a = 1, 2, \cdots , r)  \,\,\,\,  \text{ and } \\
v &= u_{g_{1}}^{(d_{1}')} \wedge u_{g_{2}}^{(d_{2}')} \wedge \cdots \wedge u_{g_{t}}^{(d_{t}')} \,\,\,\, , \,\,\,\,
g_{b} = c_{b}' - n m_{b}' \,\,\,\,,\,\,\, (b = 1, 2, \cdots , t).
\end{align*}
and suppose that $d_{a} \not= d_{b}'$ for all $a \in \{ 1, \cdots , r \}$ and $b \in \{ 1, \cdots ,t \}$. 
Then we define $\xi (u , v )$ as 

\begin{align}
\xi (u , v ) &= 
  \# \{ (a,b) \, | \, c_{a} = c_{b}'  \,\, , \,\, u_{k_{a}}^{(d_{a})} < u_{g_{b}}^{(d_{b}')}  \} .
\end{align}
\label{def-xi}
\end{definition}

\begin{lemma}[\cite{U}]
Let $a \in \mathbb{Z}$, $t \in \mathbb{Z}_{\ge 0}$, $1 \le i \le \ell$, and $1 \le j \le \ell$.

$(i)$. 
Let $u_{k}^{(j)}$ be the maximal element such that $u_{k}^{(j)} < u_{a}^{(i)}$.
Let $u_{[k,k-t]}^{(j)} = u_{k}^{(j)} \wedge u_{k-1}^{(j)} \wedge \cdots \wedge u_{k-t}^{(j)}$.
Then, 

\begin{equation*}
u_{a}^{(i)} \wedge u_{[k,k-t]}^{(j)}
=q^{- \xi (u_{[k,k-t]}^{(j)} , u_{a}^{(i)})} u_{[k,k-t]}^{(j)} \wedge u_{a}^{(i)}.
\end{equation*}

$(ii)$. 
Let $u_{g}^{(j)}$ be the minimal element such that $u_{g}^{(j)} > u_{a}^{(i)}$.
Let $u_{[g+t,g]}^{(j)} = u_{g+t}^{(j)} \wedge u_{g+t-1}^{(j)} \wedge \cdots \wedge u_{g}^{(j)}$.
Then, 

\begin{equation*}
u_{a}^{(i)} \wedge u_{[g+t,g]}^{(j)}
=q^{\xi (u_{a}^{(i)} , u_{[g+t,g]}^{(j)})} u_{[g+t,t]}^{(j)} \wedge u_{a}^{(i)}.
\end{equation*}

\label{lem3}
\end{lemma}

In the abacus presentation, $u_{a}^{(i)} , u_{[k,k-t]}^{(j)}$ and $u_{[g+t,g]}^{(j)}$ look as follows.

$$\begin{array}{cccc|crccc}
\mathrm{(i)} & \multicolumn{3}{c}{ d=i } & \multicolumn{4}{c}{ d=j } \\ \cline{8-9}
& & & &  &  & & \multicolumn{1}{|c}{ \small \centercircle{k-t} \normalsize } & \multicolumn{1}{c|}{} \\ \cline{6-7}
& & & &  & \multicolumn{1}{|c}{ \hspace{1em} \multirow{3}{*}{ $\xi  \, \left\{  \begin{array}{c} \bullet \\ \vdots \\ \bullet \end{array}   \right.$ } } & & & \multicolumn{1}{c|}{} \\ 
& & & &  & \multicolumn{1}{|c}{} & & & \multicolumn{1}{c|}{} \\
& & & &  & \multicolumn{1}{|c}{} & & \small \centercircle{k-1} \normalsize & \multicolumn{1}{c|}{ \centercircle{k} } \\  \cline{6-9}
& \centercircle{a} & &  &  &  | \hspace{1em}  & & \\
& & & &  & c \equiv a & & &  
\end{array} 
\hspace{1em} , 
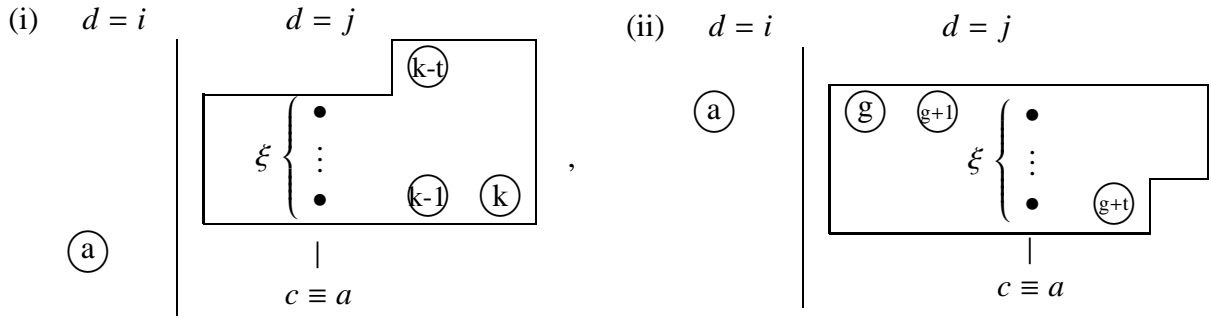
\hspace{1em}
\begin{array}{cccc|ccrcc}
\mathrm{(ii)} & \multicolumn{3}{c}{ d=i } & \multicolumn{4}{c}{ d=j } \\ 
& & & &  & & & & \\ \cline{6-9}
& \centercircle{a} & &  & &  \multicolumn{1}{|c}{ \centercircle{g} } &  \tiny \centercircle{g+1} \normalsize 
  \multirow{3}{*}{ $\xi  \, \left\{  \begin{array}{c} \bullet \\ \vdots \\ \bullet \end{array}   \right.$ }  
  & & \multicolumn{1}{c|}{\hspace{1em}} \\  
& & & &  & \multicolumn{1}{|c}{} & & & \multicolumn{1}{c|}{} \\ \cline{9-9}
& & & &  & \multicolumn{1}{|c}{} & & \multicolumn{1}{c|}{ \tiny \centercircle{g+t} \normalsize } &  \\ \cline{6-8}
& & & &  & & | \hspace{1em} & & \\ 
& & & &  & & c \equiv a & & 
\end{array}
$$
where the boxed region means that all positions are occupied by beads. 

\begin{proof}

We only show (i) by induction on $t$. 
If $t=0$, then the assertion follows from the identity (\ref{eq.22}).

Let $t \ge 1$.
Then, from the identity (\ref{eq.22}), we have

\begin{equation}
u_{a}^{(i)} \wedge u_{k-t}^{(j)} = q^{\alpha} \, u_{k-t}^{(j)} \wedge u_{a}^{(i)} + 
\sum_{ m=1}^{t} b_{m}(q) \, u_{k-t+m}^{(j)} \wedge u_{a-m}^{(i)} , 
\label{eq.28}
\end{equation}
where 
$$\alpha =
\begin{cases}
-1 & \text{ if } \,\, a \equiv k-t \text{ mod } n  \\
0 & \text{ otherwise }
\end{cases}$$ 
(see Remark (ii) after Proposition \ref{s.rule}).

Put 
$\xi = \xi (u_{[k,k-t]}^{(j)} , u_{a}^{(i)} )$ and 
$\xi' = \xi ( u_{k}^{(j)} \wedge u_{k-1}^{(j)} \wedge \cdots \wedge u_{k-t+1}^{(j)}, u_{a}^{(i)} )$. 
Then $\xi = \xi' + \alpha$, and 

\begin{align*}
u_{a}^{(i)} \wedge u_{[k,k-t]}^{(j)} 
&= u_{a}^{(i)} \wedge u_{k}^{(j)} \wedge u_{k-1}^{(j)} \wedge \cdots \wedge u_{k-t}^{(j)} \\
&=q^{\xi'} \, u_{k}^{(j)} \wedge u_{k-1}^{(j)} \wedge \cdots \wedge u_{k-t+1}^{(j)} \wedge u_{a}^{(i)} \wedge u_{k-t}^{(j)} 
\hspace{3em} \text{ (By the induction hypothesis) } \\
&=q^{\xi' + \alpha} \, u_{k}^{(j)} \wedge \cdots \wedge u_{k-t+1}^{(j)} \wedge u_{k-t}^{(j)} \wedge u_{a}^{(i)} \\
&\hspace{3em} +\sum_{m=1}^{t} q^{\xi'} b_{m}(q) \,  
u_{k}^{(j)} \wedge \cdots \wedge u_{k-t+1}^{(j)} \wedge u_{k-t+m}^{(j)} \wedge u_{a-m}^{(i)} 
\hspace{1em} \text{ (By (\ref{eq.28})) } \\
&= q^{\xi} \, u_{[k,k-t]}^{(j)} \wedge u_{a}^{(i)} 
\hspace{3em} \text{ (By Corollary \ref{cor.5}) } \\
\end{align*}
\end{proof}

\begin{definition}
Let $1 \le j \le \ell$ and $\lambda$ be a partition. 
We define 
\begin{equation*}
	\lambda^{[j]} = u_{s_{j} + \lambda_{1}}^{(j)} \wedge u_{s_{j} + \lambda_{2} -1}^{(j)} 
	\wedge u_{s_{j} + \lambda_{3} -2}^{(j)} \wedge \cdots .
\end{equation*}
In particular, 
\begin{equation*}
	\emptyset^{[j]} = u_{s_{j}}^{(j)} \wedge u_{s_{j}-1}^{(j)} \wedge u_{s_{j}-2}^{(j)} \wedge \cdots .
\end{equation*}
\label{j-partition}
\end{definition}

\begin{corollary}
Let $1 \le j \le \ell$, $r>0$, $t>0$, and put 

\begin{align*}
\emptyset^{[j]} &= u_{s_{j}}^{(j)} \wedge u_{s_{j}-1}^{(j)} \wedge \cdots \wedge u_{s_{j}-r}^{(j)} \,\,\,\, \text{ and } \,\,\,\,
u = u_{g_{1}}^{(d_{1})} \wedge u_{g_{2}}^{(d_{2})} \wedge \cdots \wedge u_{g_{t}}^{(d_{t})} \,\,\,\,.
\end{align*}

For each $b \ge 1$, let $u_{h_{b}}^{(j)}$ be the minimal element such that $u_{h}^{(j)} > u_{g_{b}}^{(d_{b})}$. 
If $d_{b} \not= j$ and $s_{j} \ge h_{b} \ge s_{j}-r $ for all $b=1,2, \cdots , t$, then 

\begin{equation*}
u \wedge \emptyset^{[j]} 
= q^{\xi ( u , \emptyset^{[j]}) - \xi (\emptyset^{[j]} , u) } \, 
\emptyset^{[j]} \wedge u .
\end{equation*}
\label{cor4}
\end{corollary}

%
%

\section{Proof of Theorem\ref{thmA}}

We only prove Theorem \ref{thmA} in the case of $\varepsilon = -$. 
The proof in the case of $\varepsilon = +$ is similar. 
Through this section we fix $j$ $(1 \le j \le \ell)$. 

\subsection{Preliminary for the proof}

Fix an sufficiently large integer $r$ so that for every ordered $q$-wedge product appearing in our argument, 
	all of the components after $r$-th factor are consecutive.
We are able to truncate $q$-wedge products at the first $r$ parts.
See \cite{U} for detail. 
Then $| \boldsymbol{\lambda} ; \boldsymbol{s} \rangle$ can be identified with $v_{\boldsymbol{\lambda}}$ defined by 
\begin{equation}
	| \boldsymbol{\lambda} ; \boldsymbol{s} \rangle = 
	v_{\boldsymbol{\lambda}} \wedge u_{s-r} \wedge u_{s-r-1} \wedge \cdots .
\label{def-v}
\end{equation}

First, we extend the definition of "sufficiently large" on the finite $q$-wedge products and introduce some notations. 

\begin{definition}
Let $u_{\boldsymbol{k}} = u_{k_{1}} \wedge u_{k_{2}} \wedge \cdots \wedge u_{k_{r}}$ 
	be an ordered $q$-wedge product and write $k_{a} = c_{a} + n (d_{a} -1) - n \ell m_{a}$ for $a=1,2,\cdots,r$ 
	as in (\ref{k->cdm}). 
Then define $\check{u_{\boldsymbol{k}}}$ to be the $q$-wedge obtained from $u_{\boldsymbol{k}}$ 
	by removing all factors $u^{(d_{a})}$ with $d_{a} = j$.
\label{checkU}
\end{definition}

\begin{lemma}
Suppose that $s_{j}$ is sufficiently large for $| \boldsymbol{\lambda} ; \boldsymbol{s} \rangle$ and 
	$\Delta^{-}_{\boldsymbol{\lambda} , \boldsymbol{\mu}}(q) \ne 0$. 
Let $v_{\boldsymbol{\lambda}} = | \boldsymbol{\lambda} ; \boldsymbol{s} \rangle$,
	$v_{\boldsymbol{\mu}} = | \boldsymbol{\mu} ; \boldsymbol{s} \rangle$ and $r$ as above.
Then, 

\begin{equation*}
\xi ( \emptyset^{[j]} , \check{v_{\boldsymbol{\lambda} }} )  = \xi ( \emptyset^{[j]} , \check{v_{\boldsymbol{\mu} }} ) .
\end{equation*}
\label{lem6}
\end{lemma}

\begin{proof}

Let $u_{k_{1}} \wedge u_{k_{2}}$ be a $q$-wedge product. 
Suppose that $s_{j}$ is sufficiently large for $u_{k_{1}} \wedge u_{k_{2}}$. 
Let $u_{k_{1}'} \wedge u_{k_{2}'}$ be a $q$-wedge product which appears 
	in the linear expansion of the straightening of $u_{k_{2}} \wedge u_{k_{1}}$.

Put $\xi = \xi ( \emptyset^{[j]} , u_{k_{1}} \wedge u_{k_{2}} )$, 
$\xi_{1} = \xi ( \emptyset^{[j]} , u_{k_{1}} )$, $\xi_{2} = \xi ( \emptyset^{[j]} , u_{k_{2}} )$, 
$\xi' = \xi ( \emptyset^{[j]} , u_{k_{1}'} \wedge u_{k_{2}'} )$, 
$\xi_{1}' = \xi ( \emptyset^{[j]} , u_{k_{1}'} )$ and $\xi_{2}' = \xi ( \emptyset^{[j]} , u_{k_{2}'} )$ 
(see Definition \ref{j-partition}). 
Note that $\xi = \xi_{1} + \xi_{2}$ and $\xi' = \xi_{1}' + \xi_{2}'$. 
Then, from the abacus presentation below, we obtain $\xi = \xi'$. 
That is, the straightening rule preserves $\xi$ if $s_{j}$ is sufficiently large. 

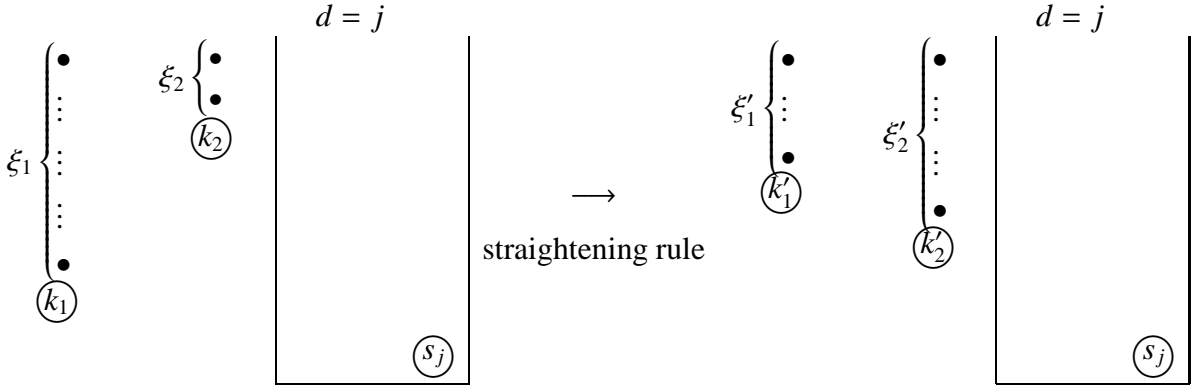
\begin{equation*}
\begin{array}{ccccccccccccc}
& & & & d=j & & & & & & & d=j & \\ 
\multirow{5}{*}{$ \xi_{1} \begin{cases} \bullet \\ \vdots \\ \vdots \\ \vdots \\ \bullet \end{cases}$} & 
 & \multirow{2}{*}{$ \xi_{2} \small \begin{cases} \bullet \\ \bullet \end{cases}$ \normalsize }
 & \multicolumn{1}{|c}{} & & \multicolumn{1}{c|}{} & 
 & \multirow{3}{*}{$ \xi_{1}' \begin{cases} \bullet \\ \vdots \\ \bullet \end{cases}$} & 
 & \multirow{4}{*}{$ \xi_{2}' \begin{cases} \bullet \\ \vdots \\ \vdots \\ \bullet \end{cases}$} 
 & \multicolumn{1}{|c}{} & &  \multicolumn{1}{c|}{} \\
& & & \multicolumn{1}{|c}{} & &  \multicolumn{1}{c|}{} & &
    & & & \multicolumn{1}{|c}{} & &  \multicolumn{1}{c|}{} \\
& & \centercircle{$k_{2}$} & \multicolumn{1}{|c}{} & &  \multicolumn{1}{c|}{} & &
    & & & \multicolumn{1}{|c}{} & &  \multicolumn{1}{c|}{} \\
& & & \multicolumn{1}{|c}{} & &  \multicolumn{1}{c|}{} & \longrightarrow & 
    \centercircle{$k_{1}'$} & & & \multicolumn{1}{|c}{} & &  \multicolumn{1}{c|}{} \\
& & & \multicolumn{1}{|c}{} & &  \multicolumn{1}{c|}{} & \text{straightening rule} & 
    & & \centercircle{$k_{2}'$} & \multicolumn{1}{|c}{} & &  \multicolumn{1}{c|}{} \\
\centercircle{$k_{1}$} & & & \multicolumn{1}{|c}{} & & \multicolumn{1}{c|}{} & &
    & & & \multicolumn{1}{|c}{} & &  \multicolumn{1}{c|}{} \\
& & & \multicolumn{1}{|c}{} & & \multicolumn{1}{r|}{\centercircle{$s_{j}$}} & &
    & & & \multicolumn{1}{|c}{} & & \multicolumn{1}{r|}{\centercircle{$s_{j}$}} \\ \cline{4-6} \cline{11-13}
\end{array} 
\end{equation*}
where beads are filled in the boxed region. 

If $\Delta^{-}_{\boldsymbol{\lambda} , \boldsymbol{\mu}} (q) \not=0 $, then 
$v_{\boldsymbol{\mu}}$ appears in the linear expansion of the straightening of $\overline{ v_{\boldsymbol{\lambda}} }$.
Therefore, the above argument assures the assertion.
\end{proof}

From Lemma.\ref{lem3}, we have 

\begin{corollary}[see \cite{U}, Lemma 5.19]
If $s_{j}$ is sufficiently large for an ordered $q$-wedge product 
$u_{\boldsymbol{k}} = u_{k_{1}} \wedge u_{k_{2}} \wedge \cdots \wedge u_{k_{r}}$.
Then

\begin{equation*}
u_{\boldsymbol{k}} = q^{- \xi ( \emptyset^{[j]} , \check{u_{\boldsymbol{k}}} )} \,
\emptyset^{[j]} \wedge \check{u_{\boldsymbol{k}}}  .
\end{equation*}
\label{cor7}
\end{corollary}

\begin{example}
Let $n=2$, $\ell =3$, $\boldsymbol{s} = (0,2,-2)$ and 
$\boldsymbol{\lambda} = ( (1,1) , \emptyset , (3) )$.
Then $s_{2}$ is sufficiently large for $ | \boldsymbol{\lambda} ; \boldsymbol{s} \rangle$.
Take $r=7$, then 

\begin{align*}
u_{\boldsymbol{k}} 
&= u_{5} \wedge u_{4} \wedge u_{3} \wedge u_{1} \wedge u_{-2} \wedge u_{-3} \wedge u_{-4} \wedge u_{-7} \\
&= u_{1}^{(3)} \wedge u_{2}^{(2)} \wedge u_{1}^{(2)} 
\wedge \underline{ u_{1}^{(1)} \wedge u_{0}^{(2)} \wedge u_{-1}^{(2)} } \wedge u_{0}^{(1)} \wedge u_{-3}^{(3)} \\
&= q^{-1} \, \underline{ u_{1}^{(3)} \wedge u_{2}^{(2)} \wedge u_{1}^{(2)} \wedge u_{0}^{(2)} \wedge u_{-1}^{(2)} }
\wedge u_{1}^{(1)} \wedge u_{0}^{(1)} \wedge u_{-3}^{(3)} \\
&= q^{-3} \, u_{2}^{(2)} \wedge u_{1}^{(2)} \wedge u_{0}^{(2)} \wedge u_{-1}^{(2)} \wedge u_{1}^{(3)} 
\wedge u_{1}^{(1)} \wedge u_{0}^{(1)} \wedge u_{-3}^{(3)} \\
&= q^{-3} \, \emptyset^{[2]} \wedge \check{ u_{\boldsymbol{k}} } .
\end{align*}
\label{ex3}
\end{example}

\begin{lemma}
If $s_{j}$ is sufficiently large for an ordered $q$-wedge product 
$u_{\boldsymbol{k}} = u_{k_{1}} \wedge u_{k_{2}} \wedge \cdots \wedge u_{k_{r}}$. 
Then, 

\begin{equation*}
\overline{ u_{\boldsymbol{k}} } = q^{- \xi ( \emptyset^{[j]} , \check{u_{\boldsymbol{k}}} )} \,
\emptyset^{[j]} \wedge \overline{ \check{u_{\boldsymbol{k}}} } .
\end{equation*}
\label{lem8}
\end{lemma}

\begin{proof}
Let $\xi = \xi ( \emptyset^{[j]} , \check{u_{\boldsymbol{k}}} )$ and 
$\eta = \xi ( \check{u_{\boldsymbol{k}}} , \emptyset^{[j]} )$. 
By Corollary \ref{cor7}, we have 

\begin{equation*}
u_{\boldsymbol{k}} = q^{- \xi} \, \emptyset^{[j]} \wedge \check{u_{\boldsymbol{k}}}  .
\end{equation*} 
Thus, we have 

\begin{align*}
\overline{ u_{\boldsymbol{k}} } 
&= q^{\xi} \, q^{- \xi -\eta} \, \overline{ \check{u_{\boldsymbol{k}}} } \wedge \overline{ \emptyset^{[j]} } 
\hspace{3em} \text{  (Definition of bar involution (\ref{barinvolution})) } \\
&= q^{ -\eta} \, \overline{ \check{u_{\boldsymbol{k}}} } \wedge \emptyset^{[j]} 
\hspace{3em} \text{  ($\overline{ \emptyset^{[j]} } = \emptyset^{[j]} $) } \\
&= q^{ -\eta} \, q^{\eta - \xi} \, \emptyset^{[j]} \wedge \overline{ \check{u_{\boldsymbol{k}}} } 
\hspace{3em} \text{  (By Corollary \ref{cor4}) } \\
&= q^{ -\xi} \, \emptyset^{[j]} \wedge \overline{ \check{u_{\boldsymbol{k}}} } 
\end{align*}
\end{proof}

\subsection{Proof of Theorem \ref{thmA}}

Let $\check{\Pi}^{\ell}$ be the subset of $\Pi^{\ell}$ whose $j$-th component is the empty Young diagram. i.e. 

\begin{equation}
	\check{\Pi}^{\ell} = \{ \boldsymbol{\lambda} \in \Pi^{\ell} \, | \, \lambda^{(j)} = \emptyset \} . 
\end{equation}

Theorem \ref{thmA} is a direct consequence of the next proposition.

\begin{proposition}
Suppose that $s_{j}$ is sufficiently large for $| \boldsymbol{\lambda} ; \boldsymbol{s} \rangle$. 
Then, 

\begin{equation*}
G^{+}(\boldsymbol{\lambda} ; \boldsymbol{s}) =
\sum_{\boldsymbol{\mu} \in \check{\Pi}^{\ell} } 
\Delta_{\check{ \boldsymbol{\lambda} } , \check{ \boldsymbol{\mu} } ;  \check{ \boldsymbol{s} } }^{+} (q) \,
| \boldsymbol{\mu} ; \boldsymbol{s} \rangle
\hspace{2em} , \hspace{2em}
G^{-}(\boldsymbol{\lambda} ; \boldsymbol{s}) =
\sum_{\boldsymbol{\mu} \in \check{\Pi}^{\ell} } 
\Delta_{\check{ \boldsymbol{\lambda} } , \check{ \boldsymbol{\mu} } ;  \check{ \boldsymbol{s} } }^{-} (q) \,
| \boldsymbol{\mu} ; \boldsymbol{s} \rangle ,
\end{equation*}
where $\check{\boldsymbol{\lambda}}$ (resp. $\check{\boldsymbol{\mu}} , \check{\boldsymbol{s}})$ 
is obtained by omitting the $j$-th component of $\boldsymbol{\lambda}$ 
(resp. $\boldsymbol{\mu} , \boldsymbol{s})$.
\label{prop11}
\end{proposition}

\begin{proof}
We only show the statement in for $G^{-}$.
The case of $G^{+}$ is treated similarly.

Take a sufficiently large integer $r$. 
Put $F = \sum_{\boldsymbol{\mu} \in \check{\Pi}^{\ell} } 
\Delta_{\check{ \boldsymbol{\lambda} } , \check{ \boldsymbol{\mu} } ;  \check{ \boldsymbol{s} } }^{-} (q) \,
| \boldsymbol{\mu} ; \boldsymbol{s} \rangle $. 
We prove $\overline{F} = F$ and $F \equiv | \boldsymbol{\lambda} ; \boldsymbol{s} \rangle$ mod 
$q^{-1} \mathcal{L}^{-}$.

The second statement is clear since 
$\check{\boldsymbol{\lambda}} = \check{\boldsymbol{\mu}}$ if and only if $\boldsymbol{\lambda} = \boldsymbol{\mu} $.
We show $\overline{F} = F$. 
Let $\xi = \xi (\emptyset^{[j]} , \check{v_{\boldsymbol{\lambda}}  })$. 

\begin{align*}
\overline{F} 
&= \sum_{\boldsymbol{\mu} \in \check{\Pi}^{\ell} } 
\Delta_{\check{ \boldsymbol{\lambda} } , \check{ \boldsymbol{\mu} } ;  \check{ \boldsymbol{s} } }^{-} (q^{-1}) \,
\overline{ u_{ \boldsymbol{\mu} } } \\
&= \sum_{\boldsymbol{\mu} \in \check{\Pi}^{\ell} } 
\Delta_{\check{ \boldsymbol{\lambda} } , \check{ \boldsymbol{\mu} } ;  \check{ \boldsymbol{s} } }^{-} (q^{-1}) \,
q^{- \xi } \, \emptyset^{[j]} \wedge \overline{ \check{ u_{ \boldsymbol{\mu} } }} 
\hspace{3em} \text{ (By Lemma \ref{lem8} \& Lemma \ref{lem6}) }  \\
&= q^{- \xi} \bigg( \sum_{\boldsymbol{\mu} \in \check{\Pi}^{\ell} } 
\Delta_{\check{ \boldsymbol{\lambda} } , \check{ \boldsymbol{\mu} } ;  
\check{ \boldsymbol{s} } }^{-} (q^{-1}) \,
\, \emptyset^{[j]} \wedge \overline{ \check{ u_{ \boldsymbol{\mu} } }} \bigg) \\
&= q^{- \xi} \, \emptyset^{[j]} \wedge 
\bigg( \sum_{\boldsymbol{\mu} \in \check{\Pi}^{\ell} } 
\Delta_{\check{ \boldsymbol{\lambda} } , \check{ \boldsymbol{\mu} } ;  \check{ \boldsymbol{s} } }^{-} (q^{-1}) \,
 \overline{ \check{ u_{ \boldsymbol{\mu} } }} \bigg) \\
&= q^{- \xi} \, \emptyset^{[j]} \wedge 
\bigg( \overline{ \sum_{\boldsymbol{\mu} \in \check{\Pi}^{\ell} } 
\Delta_{\check{ \boldsymbol{\lambda} } , \check{ \boldsymbol{\mu} } ;  \check{ \boldsymbol{s} } }^{-} (q) \,
\check{ u_{ \boldsymbol{\mu} } }} \bigg) 
\end{align*}

Note that 
$G^{-}(\check{ \boldsymbol{\lambda} } ;  \check{ \boldsymbol{s} } ) = 
\sum_{\boldsymbol{\mu} \in \check{\Pi}^{\ell} } 
\Delta_{\check{ \boldsymbol{\lambda} } , \check{ \boldsymbol{\mu} } ;  \check{ \boldsymbol{s} } }^{-} (q) \,
\check{ u_{ \boldsymbol{\mu} } }$ 
and 
$\overline{ G^{-}(\check{ \boldsymbol{\lambda} } ;  \check{ \boldsymbol{s} } ) } = 
G^{-}(\check{ \boldsymbol{\lambda} } ;  \check{ \boldsymbol{s} } ) $.
Therefore, 

\begin{align*}
\overline{F} 
&= q^{- \xi} \, \emptyset^{[j]} \wedge 
\bigg( \sum_{\boldsymbol{\mu} \in \check{\Pi}^{\ell} } 
\Delta_{\check{ \boldsymbol{\lambda} } , \check{ \boldsymbol{\mu} } ;  \check{ \boldsymbol{s} } }^{-} (q) \,
\check{ u_{ \boldsymbol{\mu} } }\bigg) \\
&= \sum_{\boldsymbol{\mu} \in \check{\Pi}^{\ell} } 
\Delta_{\check{ \boldsymbol{\lambda} } , \check{ \boldsymbol{\mu} } ;  \check{ \boldsymbol{s} } }^{-} (q) \,
v_{\boldsymbol{\mu}} 
\hspace{3em} \text{ (By Corollary \ref{cor7} \& Lemma \ref{lem6}) }  \\
&= F.
\end{align*}

\end{proof}

%
%

\section{Proof of Theorem \ref{thmB}}

Throughout this section. we fix $j$ $(1 \le j \le \ell)$.  

\subsection{The quotient space $\widetilde{\boldsymbol{F}_{q}}[\boldsymbol{s}]_{\le N}$}

In this paragraph, we fix a positive integer $N$ and assume that $s_{j}$ is sufficiently small for $N$, i.e. 
	$s_{i} - s_{j} \ge N$ for all $i \ne j$. 

We define $\widetilde{\boldsymbol{F}_{q}}[\boldsymbol{s}]_{\le N}$ to be the subspace spanned by  
$\{ \, | \boldsymbol{\lambda} ; \boldsymbol{s} \rangle \, | \, 
	\lambda^{(j)} = \emptyset \, , \, | \boldsymbol{\lambda} | \le N  \}$.
We also define a map 
$\pi \colon \boldsymbol{F}_{q}[\boldsymbol{s}] \rightarrow \widetilde{\boldsymbol{F}_{q}}[\boldsymbol{s}]_{\le N}$ 
(quotient map) by 

\begin{equation*}
\pi( \, | \boldsymbol{\lambda} ; \boldsymbol{s} \rangle ) = 
\begin{cases}
	\, | \boldsymbol{\lambda} ; \boldsymbol{s} \rangle 
		& \text{ if } \lambda^{(j)} = \emptyset  \, \text{ and } \, | \boldsymbol{\lambda} | \le N  \\
	\, 0 
		& \text{ otherwise } 
\end{cases}
\end{equation*}

We import the bar involution on $\widetilde{\boldsymbol{F}_{q}}[\boldsymbol{s}]_{\le N}$ from 
$\boldsymbol{F}_{q}[\boldsymbol{s}]$, that is 

\begin{equation}
\overline{ \pi( \, | \boldsymbol{\lambda} ; \boldsymbol{s} \rangle ) } =
\pi( \overline{ | \boldsymbol{\lambda} ; \boldsymbol{s} \rangle } ) 
\,\,, \hspace{2em}
(| \boldsymbol{\lambda} ; \boldsymbol{s} \rangle  \in \widetilde{\boldsymbol{F}_{q}}[\boldsymbol{s}]_{\le N}).
\end{equation}
The unitriangularity of the bar involution (\ref{unitriangularity}) and Lemma \ref{lem9} assure that 
the bar involution $\widetilde{\boldsymbol{F}_{q}}[\boldsymbol{s}]_{\le N}$ is well-defined.

It is clear that the following two property hold from the definition of $\widetilde{\boldsymbol{F}_{q}}[\boldsymbol{s}]_{\le N}$. 

\begin{proposition}
Let $\varepsilon \in \{ + , - \}$.
There is a unique basis 
$\widetilde{\{ G^{\varepsilon}}(\boldsymbol{\lambda} ; \boldsymbol{s}) \, 
  | \, \boldsymbol{\lambda} \in \check{\Pi}^{\ell} \,,\, | \boldsymbol{\lambda}| \le N \}$ 
of $\widetilde{ \boldsymbol{F}_{q} }[\boldsymbol{s}]_{\le N}$ such that 

\begin{align*}
\text{\rm{(i)}} \hspace{1em} & 
\overline{\widetilde{G}^{\varepsilon}(\boldsymbol{\lambda} ; \boldsymbol{s})} 
= \widetilde{G}^{\varepsilon}(\boldsymbol{\lambda} ; \boldsymbol{s})  \,\,,  \\
\text{\rm{(ii)}} \hspace{1em} & 
\widetilde{G}^{\varepsilon}(\boldsymbol{\lambda} ; \boldsymbol{s}) 
\equiv | \, \boldsymbol{\lambda} ; \boldsymbol{s} \rangle
\,\,\,\, \mathrm{mod} \,\, q^{\varepsilon} \, \widetilde{\mathcal{L}}^{\varepsilon} ,
\text{ where } \hspace{1em}
\widetilde{\mathcal{L}}^{\varepsilon} = \bigoplus_{\boldsymbol{\lambda} \in \check{\Pi^{\ell}} } 
\mathbb{Q}[q^{\varepsilon}] \, | \boldsymbol{\lambda} ; \boldsymbol{s} \rangle 
\end{align*}
and $\check{\Pi}^{\ell} = \{ \boldsymbol{\lambda} \in \Pi^{\ell} \, | \, \lambda^{(j)} = \emptyset \}$ . 
\end{proposition}

\begin{definition}
Let $\varepsilon \in \{ + , - \}$. 
Suppose that $\lambda^{(j)} = \mu^{(j)} = \emptyset$, 
	$ | \boldsymbol{\lambda} | \le N$ and $ | \boldsymbol{\mu} | \le N$. 
Define  
$\widetilde{ \Delta}^{\varepsilon} _{\boldsymbol{\lambda},\boldsymbol{\mu}}(q)$ 
by

\begin{align*}
\widetilde{G}^{+}(\boldsymbol{\mu} ; \boldsymbol{s}) = \sum_{\boldsymbol{\lambda} \in \check{\Pi}} 
\widetilde{\Delta}^{+}_{\boldsymbol{\lambda},\boldsymbol{\mu}}(q) 
  \, | \, \boldsymbol{\lambda} ; \boldsymbol{s} \rangle 
\hspace{2em} , \hspace{2em} 
\widetilde{G}^{-}(\boldsymbol{\lambda} ; \boldsymbol{s}) = \sum_{\boldsymbol{\mu} \in \check{\Pi}} 
\widetilde{\Delta}^{-}_{\boldsymbol{\lambda},\boldsymbol{\mu}}(q) \, | \, \boldsymbol{\mu} ; \boldsymbol{s} \rangle .
\end{align*}
\end{definition}

\begin{proposition}
Let $\varepsilon \in \{ + , - \}$. 
If $\lambda^{(j)} = \mu^{(j)} = \emptyset$, 
	$ | \boldsymbol{\lambda} | \le N$ and $ | \boldsymbol{\mu} | \le N$, then 

\begin{equation*}
\widetilde{\Delta}^{\varepsilon}_{\boldsymbol{\lambda},\boldsymbol{\mu}}(q) = 
\Delta^{\varepsilon}_{\boldsymbol{\lambda},\boldsymbol{\mu}}(q) .
\end{equation*}
\label{prop10}
\end{proposition}

Note that if $N \ge | \boldsymbol{\lambda} |$ and $N \ge | \boldsymbol{\mu} |$, then 
$\widetilde{\Delta}^{\varepsilon}_{\boldsymbol{\lambda},\boldsymbol{\mu}}(q)$ is independent of the choice of $N$. 

\subsection{Proof of Theorem \ref{thmB}}

As in \S4, we only prove Theorem \ref{thmB} in the case of $\varepsilon = -$.

In this paragraph, we assume that $s_{j}$ is sufficiently small for $| \boldsymbol{\lambda} ; \boldsymbol{s} \rangle$. 
Let $N = | \boldsymbol{\lambda} |$ and we fix a sufficient large integer $r$. 

The structure of our proof of Theorem \ref{thmB} is similar to that of Theorem \ref{thmA}. 
Lemma \ref{lem33} and Lemma \ref{lem44} play roles similar to Lemma \ref{lem6} and Corollary \ref{cor7} respectively. 

\begin{lemma}
Let $\boldsymbol{\lambda} , \boldsymbol{\mu} \in \Pi^{\ell}$ such that 
$\Delta^{-}_{\boldsymbol{\lambda} , \boldsymbol{\mu}} (q) \not=0 $. 
If $s_{j}$ is sufficiently small for  $| \boldsymbol{\lambda} ; \boldsymbol{s} \rangle$ 
and $\boldsymbol{\lambda}^{(j)} = \boldsymbol{\mu}^{(j)} = \emptyset $, then 

\begin{equation*}
\xi ( \emptyset^{[j]} , \check{v_{\boldsymbol{\lambda} }} )  = \xi ( \emptyset^{[j]} , \check{v_{\boldsymbol{\mu} }} ) .
\end{equation*}
\label{lem33}
\end{lemma}

\begin{proof}

Let $u_{k_{1}}^{(d_{1})} \wedge u_{k_{2}}^{(d_{2})}$ be a $q$-wedge product such that 
$d_{i} \not= j$ and $k_{i} \geq s_{j}$ for $i=1,2$.
Let $u_{k_{1}'}^{(d_{1}')} \wedge u_{k_{2}'}^{(d_{2}')}$ be a $q$-wedge product 
which appears in the linear expansion of the straightening of $u_{k_{2}} \wedge u_{k_{1}}$.

We put $\xi = \xi ( \emptyset^{[j]} , u_{k_{1}}^{(d_{1})} \wedge u_{k_{2}}^{(d_{2})} )$, 
$\xi_{1} = \xi ( \emptyset^{[j]} , u_{k_{1}}^{(d_{1})} )$, $\xi_{2} = \xi ( \emptyset^{[j]} , u_{k_{2}}^{(d_{2})} )$, 
$\xi' = \xi ( \emptyset^{[j]} , u_{k_{1}'}^{(d_{1}')} \wedge u_{k_{2}'}^{(d_{2}')} )$, 
$\xi_{1}' = \xi ( \emptyset^{[j]} , u_{k_{1}'}^{(d_{1}')} )$ and $\xi_{2}' = \xi ( \emptyset^{[j]} , u_{k_{2}'}^{(d_{2}')} )$.
Note that $\xi = \xi_{1} + \xi_{2}$ and $\xi' = \xi_{1}' + \xi_{2}'$.

Then, from the abacus presentation below, we obtain $\xi = \xi'$.

\begin{equation*}
\begin{array}{ccccccccccccccc}
& d=d_{1} & & d=d_{2} & & d=j & & & & d=d_{1} & & d=d_{2} & & d=j & \\
& \multirow{3}{*}{$\xi_{1} \begin{cases} \bullet \\ \vdots \\ \bullet \end{cases}$} & 
& \multirow{3}{*}{$\xi_{2} \begin{cases} \bullet \\ \vdots \\ \bullet \end{cases}$}
& \multicolumn{1}{|c}{} & & \multicolumn{1}{c|}{} &  & 
& \multirow{3}{*}{$\xi_{1}' \begin{cases} \bullet \\ \vdots \\ \bullet \end{cases}$} & 
& \multirow{3}{*}{$\xi_{2}' \begin{cases} \bullet \\ \vdots \\ \bullet \end{cases}$} 
& \multicolumn{1}{|c}{} & & \multicolumn{1}{c|}{} \\
& & & & \multicolumn{1}{|c}{} & & \multicolumn{1}{c|}{} &  & & & & & \multicolumn{1}{|c}{} & & \multicolumn{1}{c|}{} \\
& & & & \multicolumn{1}{|c}{} & & \multicolumn{1}{c|}{} &  & & & & & \multicolumn{1}{|c}{} & & \multicolumn{1}{c|}{} \\
& & & & \multicolumn{1}{|c}{} & & \multicolumn{1}{c|}{} & \longrightarrow & & & & & \multicolumn{1}{|c}{} & & \multicolumn{1}{c|}{} \\
\cline{5-7} \cline{13-15} 
& & & \centercircle{$k_{2}$} & & & & \text{straightening rule} & & & & & & & \\
& & & & & & &  & & \centercircle{$k_{1}'$} & & & & & \\
& & & & & & &  & & & & \centercircle{$k_{2}'$} & & & \\
& \centercircle{$k_{1}$} & & & & & &  & & & & & & & 
\end{array}
\end{equation*}
where beads are filled in the boxed region. 

Since $s_{j}$ is sufficiently small for $v_{\boldsymbol{\lambda}}$, for each $i \not= j$,  
$$ \lambda_{1}^{(i)} + s_{i} >  \lambda_{2}^{(i)} -1 + s_{i} > \cdots >  \lambda_{l}^{(i)} - l + s_{i} \ge s_{j}$$
where $l = l (\lambda^{(i)})$ is the length of $\lambda^{(i)}$.

If $\Delta^{-}_{\boldsymbol{\lambda} , \boldsymbol{\mu}} (q) \not=0 $, then 
$v_{\boldsymbol{\mu}}$ appears in the linear expansion of the straightening of $\overline{ v_{\boldsymbol{\lambda}} }$.
Therefore, the above argument assures the assertion.

\end{proof}

\begin{lemma}
Let $\boldsymbol{\lambda} \in \Pi^{\ell}$. If $\lambda^{(j)} = \emptyset$, then 

$$ v_{\boldsymbol{\lambda}} = q^{-\xi ( \check{v_{\boldsymbol{\lambda}}} , \emptyset^{[j]} )} \, 
\check{v_{\boldsymbol{\lambda}}} \wedge \emptyset^{[j]} .$$
See Definition \ref{def-xi}, Definition \ref{checkU} and Definition \ref{j-partition} 
	for the definition of $\xi$, $\check{v_{\boldsymbol{\lambda}}}$ and $\emptyset^{[j]}$ respectively.
\label{lem44}
\end{lemma}

\begin{proof}
The proof follows from Lemma \ref{lem3} and Lemma \ref{lem33}. (see also Example \ref{ex3}.)
\end{proof}

\begin{lemma}
Let $1 \le j \le \ell$, $m>0$ and $\lambda$ be a partition of length at most $m$. 
Let $1 \le d \le \ell$ and $k$ be a integer satisfying $\lambda_{1} + s_{j} \le k$.
If $j \not= d$, then $\lambda^{[j]} \wedge u_{k}^{(d)}$ is expanded as 

\begin{equation*}
\lambda^{[j]} \wedge u_{k}^{(d)} = 
q^{-\xi (u_{k}^{(d)} , \lambda^{[j]}) + \xi ( \lambda^{[j]} , u_{k}^{(d)} ) } \, 
u_{k}^{(d)} \wedge \lambda^{[j]} 
+ \sum_{ | \mu | > | \lambda | } b_{\mu}(q) u_{k- |\mu| + |\lambda|}^{(d)} \wedge \mu^{[j]} , 
\end{equation*}
where $\lambda^{[j]}$ is defined in Definition \ref{j-partition}.

Moreover, if $b_{\mu}(q) \not= 0$, then $\mu_{1} + s_{j} \le k$.
\label{lem14}
\end{lemma}

\begin{proof}
Applying the identity (\ref{eq.22}) repeatedly, we expand $\lambda^{[j]} \wedge u_{k}^{(d)}$ as a linear combination of 
$u_{k'}^{(d)} \wedge \mu^{[j]}$ such that $k' \le k$.

\end{proof}

\begin{corollary}
Let $1 \le j \le \ell$, $m>0$, $t>0$, 

\begin{align*}
\emptyset^{[j]} &= u_{s_{j}}^{(j)} \wedge u_{s_{j}-1}^{(j)} \wedge \cdots \wedge u_{s_{j}-m}^{(j)} \,\,\,\,  \text{ and } \,\,\,\,
u = u_{g_{1}}^{(d_{1})} \wedge u_{g_{2}}^{(d_{2})} \wedge \cdots \wedge u_{g_{t}}^{(d_{t})} \,\,\,\,.
\end{align*}

If $d_{b} \not= j$ for all $b=1,2,\cdots ,t$ and $s_{j}-r \le g_{1} \le g_{2} \le \cdots \le g_{t}$, 
then $\emptyset^{[j]} \wedge u$ can be written in the form  

\begin{equation*}
\emptyset^{[j]} \wedge u = 
q^{-\xi (u , \emptyset^{[j]}) + \xi ( \emptyset^{[j]} , u ) } \, u \wedge \emptyset^{[j]} 
+ \sum_{\mu \not= \emptyset} v_{\mu}(q) \wedge \mu^{[j]} ,
\end{equation*}
where $v_{\mu}(q)$ is a linear combination of $q$-wedge products.
\label{cor30}
\end{corollary}

\begin{proof}
Apply Lemma \ref{lem14} repeatedly.

\end{proof}

In the proof of Theorem \ref{thmB}, the next two lemmas (Lemma \ref{lem31} and Lemma \ref{lem45}) 
	play roles similar to Corollary \ref{cor4} and Lemma \ref{lem8} in the proof of Theorem \ref{thmA}.

\begin{lemma}

Let $\boldsymbol{\lambda} \in \Pi^{\ell}$. 
If $\lambda^{(j)} = \emptyset$, then 

\begin{equation*}
\pi ( \emptyset^{[j]} \wedge \overline{ \check{ v_{\boldsymbol{\lambda}} } }  ) = 
q^{-\xi (\check{v_{\boldsymbol{\lambda}}} , \emptyset^{[j]}) + \xi ( \emptyset^{[j]} , \check{v_{\boldsymbol{\lambda}}}) } \, 
\pi( \check{v_{\boldsymbol{\lambda}}} \wedge \emptyset^{[j]} ) . 
\end{equation*}
\label{lem31}
\end{lemma}

\begin{proof}
Let $\xi = \xi ( \emptyset^{[j]} , \check{v_{\boldsymbol{\lambda}}})$, $\eta 
= \xi (\check{v_{\boldsymbol{\lambda}}} , \emptyset^{[j]})$ and 
$$\check{ v_{\boldsymbol{\lambda}} } = 
u_{g_{1}}^{(d_{1})} \wedge u_{g_{2}}^{(d_{2})} \wedge \cdots \wedge u_{g_{t}}^{(d_{t})} . $$
From the definition of the bar involution (\ref{barinvolution}), 

$$\overline{ \check{ v_{\boldsymbol{\lambda}} } } = 
(-q)^{\kappa ( \boldsymbol{d})} \, q^{- \kappa ( \boldsymbol{c})}
u_{g_{t}}^{(d_{t})} \wedge u_{g_{t-1}}^{(d_{t-1})} \wedge \cdots \wedge u_{g_{1}}^{(d_{1})} , $$
where $(-q)^{\kappa ( \boldsymbol{d})}$ and $q^{- \kappa ( \boldsymbol{c})}$ are suitable constants (see (\ref{barinvolution})).
Then, from Corollary \ref{cor30}, 

\begin{align*}
\emptyset^{[j]} \wedge \overline{ \check{ v_{\boldsymbol{\lambda}} } } 
& = (-q)^{\kappa ( \boldsymbol{d})} \, q^{- \kappa ( \boldsymbol{c})} \,
\emptyset^{[j]} \wedge u_{g_{t}}^{(d_{t})} \wedge u_{g_{t-1}}^{(d_{t-1})} \wedge \cdots \wedge u_{g_{1}}^{(d_{1})} \\
& = (-q)^{\kappa ( \boldsymbol{d})} \, q^{- \kappa ( \boldsymbol{c})} \, q^{\eta - \xi} \,
u_{g_{t}}^{(d_{t})} \wedge u_{g_{t-1}}^{(d_{t-1})} \wedge \cdots \wedge u_{g_{1}}^{(d_{1})} \wedge \emptyset^{[j]}  
+ \sum_{\mu \not= \emptyset} v_{\mu}(q) \wedge \mu^{(j)} \\
&= q^{\eta - \xi} \, \overline{ \check{ v_{\boldsymbol{\lambda}} } } \wedge \emptyset^{[j]}  
+ \sum_{\mu \not= \emptyset} v_{\mu}(q) \wedge \mu^{[j]} ,
\end{align*}
where $v_{\mu}(q)$ is a linear combination of $q$-wedge products.

Finally, we shall prove $\pi ( v_{\mu}(q) \wedge \mu^{[j]} ) = 0$ if $\mu \not= \emptyset$. 
To do it, it is enough to prove the next claim.

\begin{claim}
Let $\mu \not= \emptyset$ and $\boldsymbol{\nu} \in \Pi^{\ell}$ such that $\nu^{(j)} = \emptyset$. 
Then, 
$$ \pi ( \check{ u_{ \boldsymbol{\nu} } } \wedge \mu^{[j]} ) = 0. $$
\end{claim}

{(\it Proof of Claim )}

Define $\boldsymbol{\nu}_{\mu} \in \Pi^{\ell}$ as 
$\boldsymbol{\nu}_{\mu}^{(j)} = \mu$ and $\boldsymbol{\nu}_{\mu}^{(i)} = \boldsymbol{\nu}^{(i)}$ ($i \not= j$). 
From the straightening rule ((\ref{eq.21}) or (\ref{eq.22})), 
any $| \boldsymbol{\rho} ; \boldsymbol{s} \rangle$ appearing 
in the linear expansion of the straightening of $\check{ u_{ \boldsymbol{\nu} } } \wedge \mu^{[j]}$ 
is less than or equal to $| \boldsymbol{\nu}_{\mu} ; \boldsymbol{s} \rangle $. 
Thus, from Lemma \ref{lem9}, the $j$-th component is not empty.
\end{proof}

\begin{lemma}
Let $\boldsymbol{\lambda} \in \Pi^{\ell}$. 
If $\lambda^{(j)} = \emptyset$, then 

$$ \pi ( \overline{ v_{\boldsymbol{\lambda}}} )  =  
q^{-\xi (\check{v_{\boldsymbol{\lambda}}} , \emptyset^{[j]})} \, 
\pi ( \overline{\check{v_{\boldsymbol{\lambda}}}} \wedge \emptyset^{[j]} ) . $$
\label{lem45}
\end{lemma}

\begin{proof}
The proof of this proposition is similarly argued to the proof of Lemma \ref{lem8}.

Let $\xi = \xi ( \emptyset^{[j]} , u_{\boldsymbol{g}})$ and $\eta = \xi (u_{\boldsymbol{g}} , \emptyset^{[j]})$.
Then, 

\begin{align*}
\overline{ v_{\boldsymbol{\lambda}} } 
&= \overline{ q^{- \xi} \, \check{ v_{\boldsymbol{\lambda}} } \wedge \emptyset^{[j]}  } \\
&= q^{\xi} \, q^{- \xi - \eta} \, \overline{ \emptyset^{[j]} } \wedge \overline{ \check{ v_{\boldsymbol{\lambda}} } } \\
&= q^{-\eta} \, \emptyset^{[j]} \wedge \overline{ \check{ v_{\boldsymbol{\lambda}} } } .
\end{align*}
Thus, from Lemma \ref{lem31}, 

\begin{align*}
\pi ( \overline{ v_{\boldsymbol{\lambda}} } )
&= q^{-\eta} \, q^{\eta - \xi} \pi ( \overline{ \check{ v_{\boldsymbol{\lambda}} } } \wedge \emptyset^{[j]}) 
= q^{- \xi} \pi ( \overline{ \check{ v_{\boldsymbol{\lambda}} } } \wedge \emptyset^{[j]}) .
\end{align*}

\end{proof}

Now Theorem \ref{thmB} is an immediate consequence of the next proposition and Proposition \ref{prop10}.

\begin{proposition}
Let $\boldsymbol{\lambda} \in \check{\Pi}^{\ell}$.
Suppose that $s_{j}$ is sufficiently small for $|\boldsymbol{\lambda} ; \boldsymbol{s} \rangle$. 
Then, 

$$ \widetilde{G}^{+}(\boldsymbol{\lambda} ; \boldsymbol{s} ) 
= \sum_{\boldsymbol{\mu} \in \check{\Pi}^{\ell} } 
\Delta^{+}_{\check{\boldsymbol{\lambda}} , \check{\boldsymbol{\mu}} ; \check{\boldsymbol{s}}}(q) 
\, \pi ( | \, \boldsymbol{\mu} ; \boldsymbol{s} \rangle ) 
\quad , \quad
\widetilde{G}^{-}(\boldsymbol{\lambda} ; \boldsymbol{s} ) 
= \sum_{\boldsymbol{\mu} \in \check{\Pi}^{\ell} } 
\Delta^{-}_{\check{\boldsymbol{\lambda}} , \check{\boldsymbol{\mu}} ; \check{\boldsymbol{s}}}(q) 
\, \pi ( | \, \boldsymbol{\mu} ; \boldsymbol{s} \rangle ) \,\,  , $$
where $\check{\boldsymbol{\lambda}}$ (resp. $\check{\boldsymbol{\mu}} , \check{\boldsymbol{s}})$ 
is obtained by omitting the $j$-th component of $\boldsymbol{\lambda}$ 
(resp. $\boldsymbol{\mu} , \boldsymbol{s})$.

In particular, 

$$ \Delta^{+}_{\check{\boldsymbol{\lambda}} , \check{\boldsymbol{\mu}} ; \check{\boldsymbol{s}}}(q) 
= \widetilde{\Delta}^{+}_{\boldsymbol{\lambda} , \boldsymbol{\mu} ; \boldsymbol{s}}(q) 
\quad , \quad 
\Delta^{-}_{\check{\boldsymbol{\lambda}} , \check{\boldsymbol{\mu}} ; \check{\boldsymbol{s}}}(q) 
= \widetilde{\Delta}^{-}_{\boldsymbol{\lambda} , \boldsymbol{\mu} ; \boldsymbol{s}}(q) .$$ 

\end{proposition}

\begin{proof}
The proof of this proposition is similarly to that of Proposition \ref{prop11}.

We only show the statement in the case of $G^{-}$.
The case of $G^{+}$ is treated similarly.

Take a sufficiently large $r$.
Put $F = \sum_{\boldsymbol{\mu} \in \check{\Pi}^{\ell} } 
\Delta_{\check{ \boldsymbol{\lambda} } , \check{ \boldsymbol{\mu} } ;  \check{ \boldsymbol{s} } }^{-} (q) \,
\pi( | \boldsymbol{\mu} ; \boldsymbol{s} \rangle ) $. 
We prove $\overline{F} = F$ and $F \equiv | \boldsymbol{\lambda} ; \boldsymbol{s} \rangle$ mod 
$q^{-1} \mathcal{L}^{-}$.

The second statement is clear since 
$\check{\boldsymbol{\lambda}} = \check{\boldsymbol{\mu}}$ if and only if $\boldsymbol{\lambda} = \boldsymbol{\mu} $.
We show $\overline{F} = F$. 
Let $\xi = \xi ( \check{v_{\boldsymbol{\lambda}} } , \emptyset^{[j]} )$. 

\begin{align*}
\overline{F} 
&= \sum_{\boldsymbol{\mu} \in \check{\Pi}^{\ell} } 
\Delta_{\check{ \boldsymbol{\lambda} } , \check{ \boldsymbol{\mu} } ;  \check{ \boldsymbol{s} } }^{-} (q^{-1}) \,
\overline{\pi( u_{ \boldsymbol{\mu} } ) } \\
&= \sum_{\boldsymbol{\mu} \in \check{\Pi}^{\ell} } 
\Delta_{\check{ \boldsymbol{\lambda} } , \check{ \boldsymbol{\mu} } ;  \check{ \boldsymbol{s} } }^{-} (q^{-1}) \,
\pi( \overline{ u_{ \boldsymbol{\mu} } } ) 
\hspace{3em} \text{ (By the definition of bar involution for $\widetilde{\boldsymbol{F}_{q}}[\boldsymbol{s}]_{\le N}$ ) } \\
&= \sum_{\boldsymbol{\mu} \in \check{\Pi}^{\ell} } 
\Delta_{\check{ \boldsymbol{\lambda} } , \check{ \boldsymbol{\mu} } ;  \check{ \boldsymbol{s} } }^{-} (q^{-1}) \,
q^{- \xi } \, \pi( \overline{ \check{ u_{ \boldsymbol{\mu} } }} \wedge \emptyset^{[j]} )
\hspace{3em} \text{ (By Lemma \ref{lem45} \& Lemma \ref{lem33}) }  \\
&= q^{- \xi} \bigg( \sum_{\boldsymbol{\mu} \in \check{\Pi}^{\ell} } 
\Delta_{\check{ \boldsymbol{\lambda} } , \check{ \boldsymbol{\mu} } ;  
\check{ \boldsymbol{s} } }^{-} (q^{-1}) \,
\,\pi( \overline{ \check{ u_{ \boldsymbol{\mu} } }} \wedge \emptyset^{[j]} ) \bigg) \\
&= q^{- \xi} \, 
\pi \left( \bigg( \sum_{\boldsymbol{\mu} \in \check{\Pi}^{\ell} } 
\Delta_{\check{ \boldsymbol{\lambda} } , \check{ \boldsymbol{\mu} } ;  \check{ \boldsymbol{s} } }^{-} (q^{-1}) \,
 \overline{ \check{ u_{ \boldsymbol{\mu} } }} \bigg) \wedge \emptyset^{[j]} \right) \\
&= q^{- \xi} \, \pi \left(
\bigg( \overline{ \sum_{\boldsymbol{\mu} \in \check{\Pi}^{\ell} } 
\Delta_{\check{ \boldsymbol{\lambda} } , \check{ \boldsymbol{\mu} } ;  \check{ \boldsymbol{s} } }^{-} (q) \,
\check{ u_{ \boldsymbol{\mu} } }} \bigg) \wedge \emptyset^{[j]} \right)
\end{align*}

Note that 
$G^{-}(\check{ \boldsymbol{\lambda} } ;  \check{ \boldsymbol{s} } ) = 
\sum_{\boldsymbol{\mu} \in \check{\Pi}^{\ell} } 
\Delta_{\check{ \boldsymbol{\lambda} } , \check{ \boldsymbol{\mu} } ;  \check{ \boldsymbol{s} } }^{-} (q) \,
\check{ u_{ \boldsymbol{\mu} } }$ 
and 
$\overline{ G^{-}(\check{ \boldsymbol{\lambda} } ;  \check{ \boldsymbol{s} } ) } = 
G^{-}(\check{ \boldsymbol{\lambda} } ;  \check{ \boldsymbol{s} } ) $.
Therefore, 

\begin{align*}
\overline{F} 
&= q^{- \xi} \, \pi \left(
\bigg( \sum_{\boldsymbol{\mu} \in \check{\Pi}^{\ell} } 
\Delta_{\check{ \boldsymbol{\lambda} } , \check{ \boldsymbol{\mu} } ;  \check{ \boldsymbol{s} } }^{-} (q) \,
\check{ u_{ \boldsymbol{\mu} } } \bigg) \wedge \emptyset^{[j]} \right) \\
&= \sum_{\boldsymbol{\mu} \in \check{\Pi}^{\ell} } 
\Delta_{\check{ \boldsymbol{\lambda} } , \check{ \boldsymbol{\mu} } ;  \check{ \boldsymbol{s} } }^{-} (q) \,
\pi( v_{\boldsymbol{\mu}} )
\hspace{3em} \text{ (By Corollary \ref{lem44} \& Lemma \ref{lem33}) }  \\
&= F.
\end{align*}

\end{proof}

%
%

\end{document}